\documentclass[12pt]{amsart}

\usepackage[utf8]{inputenc}
\usepackage[T1]{fontenc}

\usepackage[colorlinks=true, linkcolor=blue, citecolor=red]{hyperref}

\usepackage{amscd}
\usepackage{amsmath}
\usepackage{amssymb}
\usepackage{amsthm}
\usepackage{bigdelim}
\usepackage{color}
\usepackage{enumerate}
\usepackage{graphicx}
\usepackage{mathrsfs}
\usepackage{multirow}
\usepackage[all]{xy}
\usepackage{ulem}

\usepackage{amsmath,calligra,mathrsfs}

\usepackage{hyperref}
\usepackage{comment}
\newtheorem{thm}{Theorem}[section]

\newtheorem{cor}[thm]{Corollary}
\newtheorem{prop}[thm]{Proposition}
\newtheorem{conj}[thm]{Conjecture}

\newtheorem{lem}[thm]{Lemma}
\theoremstyle{definition}
\newtheorem{defn}[thm]{Definition}
\newtheorem{claim}[thm]{Claim}

\newtheorem{thmdefn}[thm]{Theorem-Definition}

\newtheorem{e.g.}[thm]{Example}
\newtheorem{ex}[thm]{Example}

\theoremstyle{remark}
\newtheorem{rem}[thm]{Remark}

\newtheorem{obs}[thm]{Observation}

\numberwithin{equation}{section}

\newcommand{\Rad}[0]{\operatorname{Rad}}
\newcommand{\Sha}[0]{\operatorname{Sha}}
\newcommand{\sha}[0]{\operatorname{sha}}

\newcommand{\codim}[0]{\operatorname{codim}}

\newcommand{\rank}[0]{\operatorname{rank}}

\newcommand{\rk}{{\rm rk}}
\newcommand{\Supp}{{\rm Supp}}

\newcommand{\Hom}{{\mathscr{H}\text{\kern -3pt {\calligra\large om}}\,}}

\newcommand{\univ}{{\rm univ}}

\newcommand{\Ker}[1]{\mathrm{Ker}(#1)}
\newcommand{\Image}[1]{\mathrm{Im}(#1)}

\newcommand{\R}{\mathbb{R}}
\newcommand{\Z}{\mathbb{Z}}

\newcommand{\C}{\mathbb{C}}
\newcommand{\Q}{\mathbb{Q}}

\title[Minimal projective manifolds with vanishing second Chern class]
{Abundance theorem \\ for minimal compact K\"ahler manifolds \\ with vanishing second Chern class}

\author{Masataka IWAI}
\address{Department of Mathematics, Graduate School of Science, the University of Osaka,
1-1, Machikaneyama-cho, Toyonaka, Osaka 560-0043, Japan.}
\email{{\tt masataka@math.sci.osaka-u.ac.jp}}
\email{{\tt masataka.math@gmail.com}}

\author{Shin-ichi MATSUMURA}
\address{Mathematical Institute 
$\&$ Division for the Establishment of Frontier Science of Organization for Advanced Studies, 
Tohoku University, 6-3, Aramaki Aza-Aoba, Aoba-ku, Sendai 980-8578, Japan.}
\email{{\tt mshinichi-math@tohoku.ac.jp}}
\email{{\tt mshinichi0@gmail.com}}

\date{\today, version 0.01}


\begin{document}

\maketitle

\begin{abstract}
In this paper, we study the abundance conjecture and the
associated Iitaka fibrations, for compact K\"ahler manifolds with
nef cotangent bundle. 
We show that the second Chern class vanishes
if and only if the cotangent bundle is nef and the canonical bundle has the numerical dimension $0$ or $1$.
Additionally, in this case, we prove that the canonical bundle is semi-ample.
Furthermore, we give a relation
between the variation of the fibers of the Iitaka fibration
and a certain semipositivity of the cotangent bundle.
\end{abstract}

\tableofcontents

\section{Introduction}

The following conjecture, the so-called \textit{abundance conjecture},
is one of the most important open problems in algebraic geometry.
The abundance conjecture has been solved for compact K\"ahler varieties of dimension at most $3$
(see \cite{Miyaoka87, Mi87, Mi88, Kawa92} for the projective case 
and \cite{CHP16, DO22, DO23} for the compact K\"ahle case). 
However, for higher-dimensional varieties,
it remains generally unsolved, and even partial results are scarce
(see \cite{DHP, GM, LP} for some approaches).

\begin{conj}[Abundance conjecture for compact K\"ahler manifolds]
\label{Abundance_Conjecture}
Let $X$ be a compact K\"ahler manifold.
If the canonical bundle $K_X$ is nef, then $K_X$ is semi-ample.
\end{conj}

In this paper, we study the abundance conjecture and the associated Iitaka fibration
for a smooth projective variety $X$ with  nef cotangent bundle,
motivated by the following problems.

\begin{itemize}
\item[(a)]  To solve the abundance conjecture for the smooth projective variety $X$
with nef cotangent bundle.
\item[(b)] To find a relation between the Fujita decomposition of $\Omega_{X}$
and the Iitaka fibration $f\colon X \to Y$ of $X$ under assuming that $K_{X}$ is semi-ample, 
where the Iitaka fibration is the semi-ample fibration induced by some high multiple of $K_{X}$. 
Recall the Fujita decomposition decomposes
nef vector bundles into numerically flat vector bundles and generically ample vector bundles
(see Proposition \ref{fujita} for the precise statement).

\item[(c)] To understand a mechanism
by which the variation of the fibers of the Iitaka fibration $f\colon X \to Y$
behaves differently depending on a certain semipositivity of curvatures of $\Omega_{X}$.

\end{itemize}

The first main result of this paper is to solve the abundance conjecture
for a compact K\"ahler manifold $X$ such that $\Omega_{X}$ is nef and $\nu(K_X) $ is $0$ or $ 1$,
equivalently,
the second Chern class $c_{2}(X)$ vanishes (see the theorems below).
Here $\nu(K_X) $ denotes the numerical dimension of the canonical bundle $K_{X}$.
Additionally, in this case,
we reveal a detailed structure of the Iitaka fibration $f\colon X \to Y$
(see Theorem \ref{Structure_compact_kahler}, cf.\,\cite{Hor13}).
We emphasize that Theorem \ref{thm-main2} is valid not only for
smooth projective varieties but also compact K\"ahler manifolds.

\begin{thm}\label{thm-main}
Let $X$ be a compact K\"ahler manifold with nef canonical bundle.
Then, the following conditions are equivalent: 
\begin{itemize}
\item[$(1)$] The second Chern class $c_2(X) $ vanishes in $H^{2,2}(X, \mathbb{R})$.
\item[$(2)$]  $\Omega_{X}$ is nef, and $\nu(K_X)$ is $0$ or $1$.
\end{itemize}
\end{thm}

\begin{thm}\label{thm-main2}
Let $X$ be a compact K\"ahler manifold with the nef cotangent bundle $\Omega_{X}$.
If $\nu(K_X) =0$ or $\nu(K_X) =1$, then the canonical bundle $K_X$ is semi-ample.
\end{thm}

Note that the abundance conjecture in the case of $\nu(K_X) =0$ has been proved in \cite{Kawa13, CGP23}. 
Let $X$ be a smooth projective variety such that $K_{X}$ is nef and $\nu(K_X)=1$. 
Theorem \ref{thm-main2} should be compared to the result of \cite{LP}, 
which states that $\kappa(K_{X}) \geq 0$ if $\chi(\mathcal{O}_{X}) \not = 0$. 
Here $\chi(\mathcal{O}_{X}) $ is the holomorphic Euler characteristic. 
In contrast to the assumption of this result, 
when $\Omega_{X}$ is nef and $\nu(K_{X}) < \dim X$, 
the Euler characteristic $\chi(\mathcal{O}_{X}) $ always vanishes (see \cite[Corollary 2.7]{DPS94}). 
Further, when $X$ is a surface, our assumption $c_2(X) = 0$ is equivalent to $\chi(\mathcal{O}_{X}) = 0$. 
In this sense, Theorem \ref{thm-main2} partially complements the result of \cite{LP}.

Wu-Zheng and Liu in \cite{WZ02, Liu14} proved that $K_{X}$ is semi-ample
under the assumption that $X$ admits a K\"ahler metric with nonpositive holomorphic bisectional curvature.
This assumption is much stronger than the condition of $\Omega_{X}$ being nef.
In fact, under this assumption,
the Iitaka fibration $f\colon X \to Y$ is locally trivial
and actually gives a product structure $X \cong F \times Y$ with the fiber $F$,
up to finite \'etale covers of $X$. 
However, when $\Omega_{X}$ is merely nef,
the Iitaka fibration $f\colon X \to Y$ (even if it exists) is not necessarily locally trivial (see \cite{Hor13}).
This phenomenon constitutes a difference
from the nonpositive holomorphic bisectional curvature
and presents difficulty in accommodating nef cotangent bundles.
We focus on the Fujita decomposition to understand the reason why such a phenomenon occurs.

The second main result gives a characterization
for the Iitaka fibration $f\colon X \to Y$ to be locally trivial
in terms of the Fujita decomposition and Hermitian flatness (see Theorem \ref{thm-main3}).

Let us explain a detailed motivation for this result.
In the case where $X$ has nonpositive holomorphic bisectional curvature,
the Ricci kernel foliation $\mathcal {R}$ can be constructed, 
and further $\mathcal {R}$ determines the Iitaka fibration $f\colon X \to Y$,
which is the core of the arguments in \cite{WZ02, Liu14}.
The Ricci kernel foliation $\mathcal {R}$ is
a certain ``flat part'' of the cotangent bundle $\Omega_{X}$.
Then, we can show that $$\mathcal {R}=\mathcal {F}=\Omega_{X/Y},$$
where $\mathcal {F}$ is the flat part of the Fujita decomposition of $\Omega_{X}$
and $\Omega_{X/Y}$ is the relative cotangent bundle of the Iitaka fibration $f\colon X \to Y$.
Unfortunately, when only the nefness of $\Omega_{X}$ is assumed,
the Ricci kernel foliation is no longer constructed.
Nevertheless, the Fujita decomposition always exists for the nef cotangent bundle $\Omega_{X}$.
Problem (b) asks when the flat part $\mathcal {F}$ of $\Omega_{X}$ might 
determine the Iitaka fibration $f\colon X \to Y$ with $\mathcal {F}=\Omega_{X/Y}$. 
The second main result (Theorem \ref{thm-main3}) answers this question.

\begin{thm}\label{thm-main3}
Let $X$ be a smooth projective variety with the nef cotangent bundle $\Omega_{X}$.
Assume that $K_{X}$ is semi-ample. 
By \cite{Hor13}, after replacing $X$ with a finite \'etale cover, 
we assume that the Iitaka fibration $f\colon X \to Y$  is an abelian group scheme. 
Then, the following conditions are equivalent: 
\begin{itemize}
\item[$(1)$]
The flat part $\mathcal {F}$ of the Fujita decomposition of $\Omega_{X}$ coincides with
$\Omega_{X/Y} $.
\item[$(2)$] The relative cotangent bundle $\Omega_{X/Y}$ is Hermitian flat.
\item[$(3)$] The Iitaka fibration $f\colon X \to Y$ is locally trivial.
\end{itemize}

\end{thm}

Assuming the abundance conjecture for a smooth projective variety $X$ with nef cotangent bundle,
H\"oring in \cite{Hor13} proved that $f\colon X \to Y$ is an abelian group scheme (up to finite \'etale covers), i.e.\,that there exists a fibration $g\colon Y \to Z$ onto a subvariety $Z$
in the fine moduli of polarized abelian varieties with level structure
so that $f\colon X \to Y$ is recovered by the fiber product of the universal family via $g\colon Y \to Z$.
From this viewpoint, the local triviality of $f\colon X \to Y$ can be rephrased as $Z$ being a one point.
Hence, the dimension $\dim Z$
can be interpreted as measuring the variation of the fibers.
Theorem \ref{thm-main3} relates the condition of $\dim Z=0$ to a certain flatness of $\Omega_{X/Y}$. 
Problem (c) asks for a more detailed understanding
of the variation of the fibers of the Iitaka fibration $f\colon X \to Y$. 
The third main result answers this question in terms of the Fujita decomposition.

\begin{thm}\label{thm-flat}
Under the same situation as in Theorem \ref{thm-main3},
we consider the Fujita decomposition of the nef cotangent bundle $\Omega_{X}$:
\begin{align*}
0 \to \mathcal{H} \to \Omega_{X} \to \mathcal{F} \to 0, 
\end{align*}
where $\mathcal{F}$ is a Hermitian flat vector bundle and
$\mathcal{H}$ is a generically ample vector bundle.

$(1)$ Let $l$ be the maximal rank of Hermitian flat subbundles of $\mathcal{H}|_{X_{y}}$,
where $\mathcal{H}|_{X_{y}}$ is the restriction of $ \mathcal{H}$
to the fiber $X_{y}:=f^{-1}(y)$ at a general point $ y\in Y$.
Then, we have $\dim Z = \rk \mathcal{H} - l$.

$(2)$ The Iitaka fibration $f\colon X \to Y$ is locally trivial
if $\Omega_{X}$ is semipositive in the following sense: 
the hyperplane bundle $\mathcal{O}_{\mathbb{P}(\Omega_{X})}(1)$
admits a smooth Hermitian metric with semipositive Chern curvature
on the projective space bundle $\mathbb{P}(\Omega_{X})$.
\end{thm}

We emphasize that the conditions of differential geometry
(e.g., Hermitian flatness or semipositivity) are related to the variation of the fibers,
which is an interesting point.
Theorem \ref{thm-flat} (2) should be compared to \cite[1.3 Theorem]{Hor13},
which states that  $f\colon X \to Y$ is locally trivial  if $\Omega_{X}$ is semi-ample.
The semipositivity of $\mathcal{O}_{\mathbb{P}(\Omega_{X})}(1)$ in Theorem \ref{thm-flat} (2)
is weaker than the condition of $X$ admitting nonpositive bisectional curvature
or $\Omega_{X}$ being semi-ample.
Therefore, under the abundance conjecture,
Theorem \ref{thm-flat} can be seen as a generalization of
\cite{WZ02, Liu14} and \cite[1.3. Theorem]{Hor13} to the semipositive case.

We finally give a technical remark on the condition $c_2(X)=0$,
which enables us to relax the assumption of Theorem \ref{thm-main}.
Note that the cotangent bundle $\Omega_{X}$ is always generically nef when $K_X$ is nef by \cite{Miyaoka87}.
As an application, we obtain a structure theorem
for a compact K\"ahler manifold $X$ such that $-K_X$ is nef and $c_2(X)=0 $. %
When $X$ is projective, Corollary \ref{Cao_Ou_structure_thm} was already proved in \cite{Cao13} and \cite{Ou17} 
by using the theory of extremal rays. %
Our method is different from their proof.

\begin{thm}
\label{generically_nef_reflexive}
Let $(X,\omega)$ be a compact K\"ahler manifold of dimension $n$, 
and $\mathcal{E}$ be a reflexive coherent sheaf on $X$. 
Assume that $\mathcal{E}$  is $\{\eta\}^{n-1}$-generically nef 
for any K\"ahler form $\eta$  
and $c_1(\mathcal{E})$ is nef.

$(1)$ There exists a positive number $ \varepsilon _{0}$ with the following property$\colon$
If $$c_2(\mathcal{E}) (c_1(\mathcal{E}) +  \varepsilon  \{\omega\})^{n-2}  = 0$$ holds for some $ \varepsilon _{0}> \varepsilon  >0$, then $\mathcal{E}$ is a nef vector bundle and  $c_2(\mathcal{E})=0$. 

$(2)$
Assume that $X$ is projective. 
If $c_2(\mathcal{E})A^{n-2}  = 0$ holds for some ample line bundle $A$, 
the sheaf $\mathcal{E}$ is a nef vector bundle and $c_2(\mathcal{E})=0$.

\end{thm}

\begin{cor}
\label{Cao_Ou_structure_thm}
Let $X$ be a compact K\"ahler manifold with the nef anti-canonical bundle $-K_X$.
If $c_2(X)=0 \in H^{2,2}(X, \mathbb{R})$,
then the tangent bundle $T_{X}$ is nef.
Moreover, there exists a finite \'etale cover $\pi\colon X' \rightarrow X$ such that $X'$ is a torus
or a $\mathbb{P}^1$-fiber bundle over a torus.
\end{cor}

This paper is organized as follows: 
In Section \ref{Sec2}, we
summarize various positivity conditions for torsion-free sheaves.
In Section \ref{Abundance_smooth}, we prove Theorems \ref{thm-main}, \ref{thm-main2}
and a structure theorem of the Iitaka fibration $f\colon X \to Y$ (Theorem \ref{Structure_compact_kahler}).
In Section \ref{sec-fujita}, we prove Theorems \ref{thm-main3} and \ref{thm-flat}.
In Section \ref{Sec5}, we prove Theorem \ref{thm-flat}.
Section \ref{section:appendix} is an appendix where we provide small remarks.

\subsection*{Acknowledgments}
The authors thank Prof.\,Wenhao Ou for helpful comments on the second Chern classes of torsion-free coherent sheaves and for giving the reference \cite{Gri10}. 
They also thank Prof.\,Xiaojun Wu for kindly sharing his ideas regarding Theorem \ref{thm-Torus-fibration}. 
They would like to thank the referee for valuable comments, particularly those concerning the reflexivity assertions in Theorems~\ref{Ou_chapter6} and~\ref{Cao13_revisit}. 

M.\,I. thanks Prof.\,Masaru Nagaoka for discussions and helpful comments.
He was partially supported by Osaka City University, Advanced Mathematical Institute (MEXT Joint Usage/Research Center on Mathematics and Theoretical Physics, JPMXP0619217849) and
Grant-in-Aid for Early Career Scientists $\sharp$22K13907.
He also thanks the members in RACMaS -Research Alliance Center for Mathematical Sciences-,
Tohoku University.

S.\,M. thanks Prof.\,Xiaojun Wu for discussions about special varieties.
This paper was written during his stay at the University of Bayreuth.
He also thanks the members of the University of Bayreuth for their hospitality.
He was partially supported
by Grant-in-Aid for Scientific Research (B) $\sharp$21H00976, 
Fostering Joint International Research (A) $\sharp$19KK0342 from JSPS, 
and JST FOREST Program $\sharp$JPMJFR2368.

\section{Preliminaries}\label{Sec2}

\subsection{Notation}

Throughout this paper, we work over the field of complex numbers,
and basically follow the notation of \cite{Har77, KM98}.
We interchangeably use the words ``Cartier divisors,'' ``invertible sheaves,'' and ``line bundles,''
and we also identify ``locally free sheaves of finite rank'' with ``vector bundles.''
We assume that all the sheaves are coherent without explicit mention.
For a torsion-free sheaf $\mathcal{F}$ on a normal variety $X$,
we define the dual sheaf $\mathcal{F}^{*}$ and the determinant sheaf $\det \mathcal{F}$ 
by
$$
\mathcal{F}^{*}:= \Hom_{\mathcal{O}_{X}}(\mathcal{F}, \mathcal{O}_{X})
\quad \text{and} \quad
\det \mathcal{F}:= (\Lambda^{\rk \mathcal{F}} \mathcal{F})^{**}.
$$

\subsection{Positivity, flatness, and slope of torsion-free sheaves.}

In this subsection, we summarize some properties defined for vector bundles or torsion-free sheaves.
Throughout this subsection,
let $X$ be a compact K\"ahler manifold of dimension $n$,
and let $E$ be a vector bundle of rank $r$ on $X$.

\begin{defn} [{\cite[Proposition 1.11, Definitions 1,7, 1.9]{DPS94},
\cite[Propositions 2.5, 4.21, 4.22, Corollary 2.7]{Kobayashi}, \cite[Definition 2.1]{Pet1}}]
A vector bundle $E$ is said to be
\begin{enumerate}
\item  \textit{Griffith semipositive} if $E$ admits a smooth Hermitian metric
with Griffith semipositive curvature;
\item \textit{nef} if the hyperplane bundle $\mathcal{O}_{\mathbb{P}(E)}(1)$ is a nef line bundle
on the projective space bundle $\mathbb{P}(E)$;
\item \textit{semipositive} if $\mathcal{O}_{\mathbb{P}(E)}(1)$ admits a smooth Hermitian metric
with a semipositive Chern curvature on the projective space bundle $\mathbb{P}(E)$;

\item \textit{Hermitian flat} if $E$ is Griffiths semipositive and $c_1(E)=0$;
\item  \textit{numerically flat} if $E$ is nef and $c_1(E)=0$;
\item  \textit{flat} if $E$ admits a  connection $\nabla$ whose curvature vanishes;
\item  \textit{projectively Hermitian flat} if $E$ admits a smooth Hermitian metric $h$ such that
the Chern curvature tensor $\Theta_{E,h}$ satisfies $\Theta_{E,h} = \alpha {\rm{Id}}_{E}$ for some 2-form $\alpha$;
\item  \textit{projectively flat} if $E$ admits a  connection $\nabla$ such that
$\nabla^2 = \alpha {\rm{Id}}_{E}$ for some 2-form $\alpha$, 
equivalently, there exists a representation $$\rho \colon \pi_1 (X)  \rightarrow \mathbb{P}GL(r, \C)$$ such that $\mathbb{P}(E) \cong X_{\univ} \times_{\rho} \mathbb{P}^{r -1}$,
where $X_{\univ}$ is the universal cover of $X$.
\end{enumerate}

We further assume that $X$ is a smooth projective variety.
Let $A_1, \ldots, A_{n-1}$ be ample line bundles on $X$.
Then $E$ is said to be
\begin{enumerate}
\item [(9)] $(A_1, \ldots, A_{n-1})$-\textit{generically ample} (resp.\,$(A_1, \ldots, A_{n-1})$-\textit{generically nef})
if $E|_{C}$ is ample (resp.\,nef) for a curve $C := D_{1} \cap \dots \cap D_{n-1}$
defined by general members $D_{i} \in | m_i A_i|$ and $m_i \gg 0 $. 
See Lemma~\ref{lem-ch} for a characterization in terms of slopes
and Definition~\ref{def-g-ample} for the definition for compact K\"ahler manifolds. 
\end{enumerate}
The vector bundle $E$ is simply called \textit{generically ample} (resp.\,\textit{generically nef})
if $E$ is $(A_1, \ldots, A_{n-1})$-generically ample (resp.\,$(A_1, \ldots, A_{n-1})$-\textit{generically nef})
for any ample line bundles $\{A_i\}_{i=1}^{n-1}$.
When $X$ is only assumed to be a compact K\"ahler manifold, 
we define the generic ampleness and generic nefness by using K\"ahler classes instead of ample line bundles.

\end{defn}

We now discuss slopes, flatness, and stability of torsion-free sheaves. 
Let $\mathcal{E}$ be a torsion-free sheaf of rank $r$ on $X$ and
$\alpha \in H^{n-1, n-1}(X, \mathbb{R})$ be a movable class on $X$. 
Here, a class $\alpha \in H^{n-1,n-1}(X,\mathbb{R})$ is said to be 
\textit{movable} if it belongs to the closure of the convex cone generated
by strongly movable classes. 
A \textit{strongly movable} class is a class of the form  
$$
\pi_*\bigl( \{ \omega_1 \}\wedge \cdots \wedge \{\omega_{n-1}\}\bigr),
$$
where $\pi \colon X' \to X$ is a bimermorphic morphism 
and $\omega_1,\ldots,\omega_{n-1}$ are K\"ahler classes on $X'$.

The {\it slope $\mu_{\alpha}(\mathcal{E})$ of $\mathcal{E}$ with respect to $\alpha$ } is defined by
$$
\mu_{\alpha}(\mathcal{E}):= \frac{c_1(\mathcal{E})  \alpha }{r}
\text{ \quad and \quad } c_1(\mathcal{E}):= c_{1}(\det \mathcal{E}).
$$
The maximal slope $\mu_{\alpha}^{\max}(\mathcal{E})$
(resp.\,the minimal slope $\mu_{\alpha}^{\min}(\mathcal{E})$)
is defined
by the supremum of $\mu_{\alpha}(\mathcal{S})$ for any nonzero subsheaf
$\mathcal{S} \subset \mathcal{E}$
(resp.\,the infimum of $\mu_{\alpha}(\mathcal{Q})$
for any nonzero torsion-free quotient sheaf $\mathcal{E} \to \mathcal{Q} \to 0$).
The sheaf $\mathcal{E} $ is said to be $\alpha$-{\it semistable with respect to $\alpha$}
if $\mathcal{E}=\mathcal{E}_{\max}$.
Here $\mathcal{E}_{\max} \subset \mathcal{E}$, which is called the {\it maximal destabilizing sheaf},
is the maximal subsheaf with respect to the inclusion such that
$\mu_{\alpha}(\mathcal{E}_{\max}) = \mu_{\alpha}^{\max}(\mathcal{E}) $,
which uniquely exists (see \cite[Lemma 7.17]{Kobayashi}).

\begin{thmdefn}
[{\cite[Chapter 4. Theorem 4.1]{Nakayama}, \cite[Proposition 13, Corollary 10]{Wu20}, \cite[Definition 4.1]{LOY20}}]
\label{numerically_projectively_flat}
A reflexive sheaf $\mathcal{E}$ on $X$ is said to be \textit{projectively numerically flat}
if it satisfies one of the equivalent following conditions: 
\begin{enumerate}
\item $\mathcal{E}$ is locally free, and the $\Q$-twisted vector bundle $\mathcal{E} \langle \frac{\det \mathcal{E}^{*}}{r}\rangle $ is nef.
\item $\mathcal{\mathcal{E}}$ is $\{\omega\}^{n-1}$-semistable
and satisfies that
$$
\left( c_2(\mathcal{E}) - \frac{r-1}{2r}c_1(\mathcal{E})^{2}\right)  \{\omega\}^{n-2}=0
$$
for some K\"ahler form $\omega$.
\item $\mathcal{E}$ is locally free, and there exists a filtration of subbundles:  
$$
0=\colon E_0 \subset E_1 \subset \cdots \subset E_l:= \mathcal{E}
$$
such that $G_i := E_i / E_{i-1}$ is a projectively Hermitian flat vector bundle and
$c_1(G_i)/ \rk G_i = c_1(E)/r \in H^{1,1}(X,\R)$ holds for any $i = 1, \ldots, l$.
\end{enumerate}
\end{thmdefn}
\noindent
These relations can be summarized by the following table.
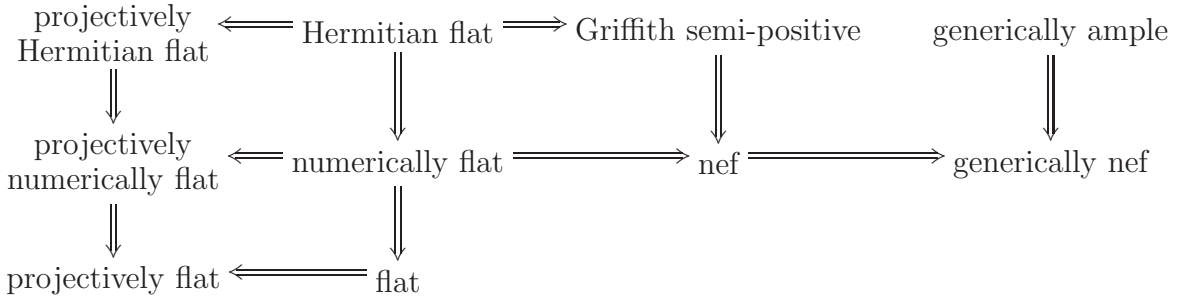
\begin{equation*}
\xymatrix@C=20pt@R=20pt{
\txt{projectively \\Hermitian  flat} \ar@{=>}[d]&    \txt{Hermitian  flat} \ar@{=>}[d]\ar@{=>}[r]\ar@{=>}[l]& \txt{Griffith semi-positive}\ar@{=>}[d]&  \txt{generically ample}\ar@{=>}[d]\\
\txt{projectively \\ numerically flat}  \ar@{=>}[d]& \txt{numerically flat} \ar@{=>}[d]\ar@{=>}[r]\ar@{=>}[l]&\txt{nef}\ar@{=>}[r]&\txt{generically nef}  \\ 
\txt{projectively flat}&  \txt{flat}\ar@{=>}[l]& & 
}
\end{equation*}

The notion of generically ample vector bundles with respect to ample line bundles can be characterized by the slope,
which enables us to extend the definition
to any movable classes on compact K\"ahler manifolds.

\begin{lem}\label{lem-ch}
Let $X$ be a smooth projective variety of dimension $n$.
Let $A_1, \ldots, A_{n-1}$ be ample line bundles on $X$.
Then, a torsion-free sheaf $\mathcal{E}$ is $(A_1, \ldots, A_{n-1})$-generically ample
$($resp.\,$(A_1, \ldots, A_{n-1})$-generically nef$)$
if and only if 
$$
\mu^{\min}_{A_1\cdots A_{n-1}}(\mathcal{E}) > 0 \quad (\text{resp}.\,\mu^{\min}_{A_1\cdots A_{n-1}}(\mathcal{E}) \ge 0). 
$$
\end{lem}

\begin{defn}\label{def-g-ample}
Let $\alpha$ be a movable class on a compact K\"ahler manifold $X$.
Then, a torsion-free sheaf $\mathcal{E}$ is said to be $\alpha$-\textit{generically ample}
(resp.\,$\alpha$-\textit{generically nef}) if $\mu^{\min}_{\alpha}(\mathcal{E}) > 0$ (resp.\, $\mu^{\min}_{\alpha}(\mathcal{E}) \ge 0$).
\end{defn}

\begin{proof}[Proof of Lemma \ref{lem-ch}]
Assume that $\mathcal{E}$ is $(A_1, \ldots, A_{n-1})$-generically ample.
Let
$$
0 =\colon \mathcal{E}_0 \subset \mathcal{E}_1 \subset \cdots \subset \mathcal{E}_l:= \mathcal{E}.
$$
be the Harder-Narasimhan filtration with respect to this polarization.
Then, straightforward computation yields
$$
\mu^{\min}_{A_1\cdots A_{n-1}}(\mathcal{E}) = \mu_{A_1\cdots A_{n-1}}(\mathcal{E}_{l} /\mathcal{E}_{l-1}).$$
By definition, we have
$$
\mu_{A_1\cdots A_{n-1}}(\mathcal{E}_{l} /\mathcal{E}_{l-1}) =
\frac{1}{m_1 \cdots m_{n-1} \cdot \rank (\mathcal{E}_{l} /\mathcal{E}_{l-1})} \deg_{C}( (\mathcal{E}_{l} /\mathcal{E}_{l-1} |_{C}),
$$
where $C := D_{1} \cap \dots \cap D_{n-1}$ is a curve defined by
general members $D_{i}$ in $| m_i A_i|$ and $m_i \gg 0 $.
Then, the right-hand side is positive since $\mathcal{E}|_{C}$ is ample.

The converse implication follows from Mehta-Ramanathan's theorem.
The restriction of the Harder-Narasimhan filtration to $C$ constructed above
also gives a Harder-Narasimhan filtration of $\mathcal{E}|_{C}$.
Hence, we obtain
$$
0<\mu^{\min}_{A_1\cdots A_{n-1}}(\mathcal{E})=\mu_{A_1\cdots A_{n-1}}(\mathcal{E}_{l} /\mathcal{E}_{l-1})
=\mu(\mathcal{E}_{l} /\mathcal{E}_{l-1} |_{C})=\mu^{\min}(\mathcal{E}|_{C}), 
$$
which finishes the proof.
The proof for the nef case is almost the same, and thus we omit it.
\end{proof}

After we recall our definition of pseudoeffective vector bundles, 
we confirm the following lemma, which is easily proved, but useful in this paper. 
\begin{defn}
A vector bundle $E$ on a compact complex manifold $X$ is said to be {\textit{pseudoeffective}}
if the nonnef locus of $\mathcal{O}_{\mathbb{P}(E)}(1)$ is not dominant over $X$.
Note that this definition requires a stronger condition than the pseudoeffectivity of 
$\mathcal{O}_{\mathbb{P}(E)}(1)$ (see \cite[Proposition 2.4]{Mat23} for the details).
\end{defn}

\begin{lem}
\label{nef_generically_ample}
Let $X$ be a compact K\"ahler manifold and $E$ be a vector bundle on $X$.
\begin{enumerate}
\item[$(1)$] If $E$ is pseudoeffective, then $E$ is $\alpha$-generically nef for any movable class $\alpha$.
\item[$(2)$] If $E$ is pseudoeffective and $\mu^{\min}_{\beta}(E) >0$ for some movable class $\beta$, then $E$ is $\alpha$-generically ample for any class $\alpha$ in the interior of the movable cone.
\end{enumerate}
\end{lem}
\begin{proof}
(1). For any torsion-free quotient $E \rightarrow Q$, the Chern class $c_1(Q)$ is pseudoeffective,
and thus $\mu^{\min}_{\alpha}(E) \geq 0$ holds for any movable class $\alpha$.
\\
(2). Assume that there exists a class $\alpha$ in the interior of the movable cone
such that $\mu^{\min}_{\alpha}(E) \le 0$.
By definition, we can find a torsion-free quotient $E \rightarrow Q$ such that
$\mu^{}_{\alpha}(Q) \le 0$.
Then, we obtain $c_1(Q)=0$ since $\det Q$ is pseudoeffective 
by the same argument as in \cite[Lemma 2.1]{Mat13}. 
On the other hand, by assumption,
we obtain $\mu^{\min}_{\beta}(E) \le  \mu_{\beta}(Q) =0$,
which is a contradiction.
\end{proof}
Note that the converse of (1) in Lemma \ref{nef_generically_ample} 
does not hold even for cotangent bundles.
Indeed, by \cite{CP11},
the cotangent bundle $\Omega_{X}$ is generically nef
when $K_{X}$ is nef.
Therefore, the cotangent bundle of a (strict) Calabi-Yau manifold is generically nef, 
but it is not pseudoeffective by \cite[Corollary 6.12]{DPS01}.

\section{Abundance Theorem and Structure Theorem}
\label{Abundance_smooth}
\subsection{Abundance theorem}

This subsection is devoted to the proofs of Theorems \ref{thm-main} and \ref{thm-main2}.
We first prove Theorem \ref{thm-main}, assuming Theorem \ref{generically_nef_reflexive}.
Theorem \ref{generically_nef_reflexive} is proved in Section \ref{Sec5}.

\begin{proof}[Proof of Theorem \ref{thm-main}]
Assume Condition (2).
Then, from \cite[Corollary 2.6]{DPS94}, we obtain
$$
0 \leq c_{2}(X) \omega^{n-2} \leq c_{1}(X)^{2} \omega^{n-2} =0
$$
for any K\"ahler form $\omega$.
This implies that $c_{2}(X) =0$ by Theorem \ref{generically_nef_reflexive} (1). 

Conversely, we assume Condition (1).
Since $K_{X} $ is nef,
the cotangent bundle $\Omega_{X}$ is $\{\omega\}^{n-1}$-generically nef for any K\"ahler from $\omega$ 
by \cite[Theorem 1.4]{Enoki93} and \cite[Theorem 1.2]{Cao13}.
Then, by Theorem \ref{generically_nef_reflexive} (1) and $c_{2}(X)=0$,
we can see that $\Omega_{X}$ is nef.
It remains to show that $\nu(K_X) \leq 1$.
If $\nu(K_X) > 1$, 
there exists a nef line bundle $L \subset \Omega_{X}$ such that $c_1(L)=c_1(K_X)$.
This follows from \cite{Cao13} in the K\"ahler case and \cite{Ou17} in the projective case
(see Theorems \ref{Cao13_revisit} and \ref{Ou_chapter6}).
Then, the Bogomolov-Sommese vanishing theorem implies 
$H^{0}(X, \Omega_{X} \otimes L^{*})=0$ 
(see \cite{Mou98, Wu20-Bogomolov} for the compact K\"ahler case).
This contradicts the fact that $L$ is a subbundle of $\Omega_{X}$.
Hence, we conclude that $\nu(K_X) =0$ or $1$.
\end{proof}

For the proof of Theorem \ref{thm-main2},
we focus on the notion of special manifolds introduced by Campana in \cite{Cam04}.
A compact K\"ahler manifold $X$ is said to be \textit{special}
if any $p \in \mathbb{Z}_{+}$ and any saturated subsheaf $F \subset  \Omega_{X}^{p}$ of rank $1$ satisfies $\kappa(F) < p$,
where $\kappa(F)$ denotes the Kodaira dimension of $F$.
In general, it is quite difficult to determine whether a given manifold is special or nonspecial. 
Nevertheless, the following theorem gives a criteria of manifolds to be nonspecial.

\begin{thm}
\label{not_special}
Let $X$ be a compact K\"ahler manifold.
If there exists a projectively flat subsheaf $E \subset \Omega_{X}$ 
such that $c_1(E)$ is nef and $\nu(c_1(E))=1$,
then $X$ is not special.
\end{thm}

\begin{proof}
Any finite \'etale cover of special manifolds is special by \cite[Theorem 5.12]{Cam04}, 
and thus we can freely replace $X$ with its finite \'etale covers. 
By the definition of  projective flatness, there exists a representation
$$
\rho\colon \pi_1 (X)  \rightarrow \mathbb{P}GL(r, \C)
$$
such that $\mathbb{P}(E) \cong X_{\univ} \times_{\rho} \mathbb{P}^{r-1}$, where $r:=\rk E$.

In the case where $\Image{\rho}$ is finite, the kernel $\Ker{\rho} \subset \pi_{1}(X)$ is of finite index. 
Hence, after replacing $X$ with the finite \'etale cover corresponding to $\Ker{\rho}$,
we may assume that $\Image{\rho}$ is a one point. 
This implies that $E \cong L^{\oplus r}$ holds for some line bundle $L$.
We have $\nu(L)=\nu(c_{1}(E))=1$ by assumption, and hence $X$ is not special by \cite[Theorem A]{PRT21}.

From now on, we assume that $\Image{\rho}$ is infinite.
Let $H$ be the algebraic Zariski closure of $\Image{\rho}$.
Note that $H$ has finitely many connected components since $H$ is a linear algebraic group.
After replacing $X$ with its finite \'etale cover, we may assume that $H$ is connected.
For the solvable radical $\Rad(H)$ of $H$,
we consider
$$
\rho'\colon \pi_1(X) \xrightarrow{\rho} H \twoheadrightarrow H/\Rad(H).
$$
and the kernel $K:={\rm Ker}(\rho')$.
By Selberg's Lemma, we may assume that the image of $\rho'$ is torsion-free.
We now take a $K$-Shafarevich map $\sha^{K} \colon X \dashrightarrow \Sha^{K}(X)$
(see \cite[Theorem 3.6]{Kol95} or \cite[Chapter 3]{Cam94} for Shafarevich maps.)

In the case of $\dim \Sha^{K}(X)=0$, after replacing $X$ with its finite \'etale cover,
we may assume that the image of $\rho$ is contained in $ \Rad(H)$.
Then, by the argument of \cite[Proposition 3.3]{GKP20b},
we can conclude that
$$
E\cong A \otimes G,
$$
where $A$ is a line bundle and $G$ is a flat vector bundle admitting a filtration
$$
0=G_0 \subset G_1 \subset \cdots \subset G_r = G
$$
such that $G_i$ is a flat vector bundle and $G_1 \cong \mathcal{O}_X$.
From $G_1 \cong \mathcal{O}_X$, we obtain
$$A = G_1 \otimes A \subset G \otimes A =E \subset \Omega_{X}.$$
Hence, from $\nu(A)=\nu(c_{1}(E))=1$,
we can conclude that $X$ is not special by using \cite[Theorem A]{PRT21} again.

When $\dim \Sha^{K}(X)>0$,
the base variety $\Sha^{K}(X)$ of the $K$-Shafarevich map is of general type since $H / \Rad(H)$ is semi-simple and the image of $\rho'$ is torsion-free,
which follows from \cite[Th\'eor\`eme 1]{CCE15}.
Then, the manifold $X$ is not special by the definition of special manifolds.
\end{proof}

The theorem below ensures that nonspecial manifolds admit
a fibration of general type in the sense of Campana
(see \cite{Cam04, Cam11} for the details).

\begin{thm}[{\cite[Theorems 2.25, 2.26 and Corollaries 1.13, 4.7]{Cam04}}]
\label{campana_special}
Let $X$ be a compact K\"ahler manifold.
Assume $X$ is not special.
Then, there exists an almost holomorphic map $f\colon X \dashrightarrow Y$,
called {\textit{the core fibration}},
onto a compact K\"ahler manifold $Y$ with $0 < \dim Y \le \dim X$
satisfying the following diagram and properties$\colon$
\begin{equation*}
\xymatrix@C=30pt@R=20pt{
X'\ar@{->}[d]_{\pi_{X}} \ar@{->}[r]^{f'}&Y' \ar@{->}[d]^{\pi_{Y}}\\ 
X\ar@{-->}[r]^{f} &  Y, \\
 }
\end{equation*}

\begin{itemize}
\item[(1)]
Both $\pi_{X} \colon X' \rightarrow X$ and $\pi_{Y} \colon Y' \rightarrow Y$ are bimeromorphic maps between compact K\"ahler manifolds.

\item[(2)] A general fiber $F'$ of $f'\colon X' \to Y'$ is special.

\item[(3)] $\kappa(K_{Y'} + \Delta(f')) = \dim Y'$ holds, where $\Delta(f')$ is the multiplicity divisor of $f'$.

\item[(4)]
$
\kappa(K_{X'}) =\kappa(K_{F'}) +\dim Y'
$ holds
for a general fiber $F'$ of $f'\colon X' \to Y'$.
\end{itemize}
\end{thm}

Under the above preparation, we prove Theorem \ref{thm-main2}.

\begin{proof}[Proof of Theorem \ref{thm-main2}]
Note that the abundance conjecture in the case of $\nu(K_X) =0$ has been proved in \cite{Kawa13, CGP23}. 

We first consider the case of $\nu(K_X) =0$. 
In this case, the abundance conjecture has been proved in a more general situation 
(see \cite{Kawa13, CGP23}). 
Even without this, in our situation, the cotangent bundle $\Omega_{X}$ is numerically flat, 
and thus, there exists a finite \'etale cover $T \rightarrow X$ from a torus $T$.
This implies that $K_{X}$ is semi-ample. 

We consider the case of $\nu(K_X) =1$. 
In this case, we prove the semi-ampleness of $K_{X}$ by induction on $n:= \dim X$.
Note that we have $c_{2}(X)=0$ by Theorem \ref{thm-main}. 
Hence, by applying Theorem \ref{Cao13_revisit}  and the remark below (see Theorem \ref{Ou_chapter6} 
and Remark \ref{rem-filtration-subbundle} in the projective case),
we can find a projectively numerically flat subbundle
$E \subset \Omega_{X}$ with $\nu(c_1(E)) =1$. 
The subbundle $E$ is projectively flat by \cite[Theorem 4.6]{LOY20},
and thus $X$ is not special by Theorem \ref{not_special}.
Then, by Theorem \ref{campana_special}, there exists a nontrivial core fibration $f\colon X \dashrightarrow Y$.
Note that $ \kappa(K_{X'}) = \kappa(K_{X}) $ and $\kappa(K_{F'}) =\kappa(K_{F})$
since $X$ and $F$ are smooth,
where $F$ is a general fiber of $f\colon X \dashrightarrow Y$.
This implies that
$$
\kappa(K_{X}) =\kappa(K_{F}) +\dim Y.
$$
It is sufficient to show that $\kappa(K_{F})\not = -\infty$.

The cotangent bundle $\Omega_{F}$ is nef,
and $$\nu(K_{F}) = \nu(K_{X}|_{F}) \leq \nu(K_{X}) = 1$$ holds.
Hence, by the induction hypothesis, the canonical bundle $K_{F}$ is semi-ample. 
In particular, we have $\kappa(K_{F})\geq 0$.
This implies that
$$
1= \nu(K_X) \ge  \kappa(K_{X}) =\kappa(K_{F}) +\dim Y  \ge \dim Y \ge 1.
$$
Thus, the canonical divisor $K_X$ is semi-ample from Kawamata's theorem 
(see \cite[Theorem 1.1]{Kaw85}, \cite[Corollary 5.6]{Nak87}, and \cite[Theorem 4.8]{Fuj11}). 
\end{proof}


\subsection{Structure theorem}

In this subsection, with specific examples,
we examine properties of the Iitaka fibration $f\colon X\to Y$ of
$X$ with nef cotangent bundle.
When $X$ is a smooth projective variety,
H\"oring in \cite{Hor13} showed that, up to finite \'etale covers,
the Iitaka fibration $f\colon X\to Y$ (if it exists) is an abelian group scheme
onto a smooth projective variety $Y$.
In particular, the Iitaka fibration $f\colon X\to Y$ always admits a section $s\colon Y \to X$.
Furthermore, the base variety $Y$ has the nef cotangent bundle and 
 ample canonical bundle $K_{Y}$.
However, by Example \ref{ex-ara}, we see  that
the Iitaka fibration does not necessarily have a section when we assume only that $X$ is K\"ahler.
This comes from the lack of fine moduli of complex tori. 
Further, by Example \ref{ex-cam}, we see that 
the Iitaka fibration $f\colon X \to Y$ is not smooth and $K_{Y}$ is not  ample, 
without taking finite \'etale covers even if $X$ is projective.

\begin{ex}\label{ex-ara}
This example was given by Donu Arapura in Mathoverflow \cite{Ara}.
Let $C$ be a smooth curve of the genus $g \geq 2$ curve and $E$ be an elliptic curve.
Fixing a nontorsion point $p \in E$ and the isomorphism $H_{1}(C, \Z) \cong \Z^{2g}$,
we define a homomorphism
$$
\begin{array}{rccc}
&  \tilde{h} \colon H_{1}(C, \Z) \cong \Z^{2g}                   &\rightarrow& E                 \\
        & (x_1, \ldots, x_{2g})                  & \longmapsto   & (\sum_{i=1}^{2g} x_i )p. 
\end{array}
$$
By composing the natural morphism $\pi_{1}(C) \rightarrow  H_{1}(C, \Z)$ and $\tilde{h}$,
we obtain a homomorphism $h \colon \pi_{1}(C) \rightarrow E$.
Note that the image of $h$ is dense in $E$ since $p \in E$ is a nontorsion point.

Let $C_{\univ}$ be a universal cover of $C$.
We define an action of $\pi_{1}(C)$ on $ C_{\univ} \times E$
to be $$\gamma(x,y) = (\gamma^{-1}x, y + h(\gamma))$$ for $\gamma \in \pi_{1}(C)$.
Let $X$ be the quotient of $C_{\univ} \times E$ by this action.
Then the induced morphism $f \colon X \rightarrow C$ by projection gives the Iitaka fibration of $X$.
The Iitaka fibration $f \colon X \rightarrow C$ does not admit a section even if we replace $X$ with finite \'etale covers.
Indeed, if the Iitaka fibration $X' \to C'$ of a finite \'etale cover $X' \to X$ admits a section,
the image of the section gives a relatively ample divisor on $X$. 
Therefore, we can conclude that $X'$ (and thus $X$) is projective.
However, the elliptic curve $E$ does not have
a nontrivial
line bundle that is invariant by the action of $h \colon \pi_{1}(C) \rightarrow E$.
This implies that $X$ is not projective.

\end{ex}

\begin{ex}[{\cite[Example 2.12]{Cam11}}]\label{ex-cam}
This example shows that the Iitaka fibration $f \colon X \to Y$ is not smooth
and that the base variety $Y$ of is not of general type
without taking finite \'etale covers of $X$.

Let $E$ be an elliptic curve with a translation $\tau$ of order $2$
and $C$ be a hyperelliptic curve of genus $g \ge 2$
with an involution $i \colon C \rightarrow C$.
Then, the group $G :=\Z/2\Z$ acts on $C \times E$ by $i$ and $\tau$.
Let us consider the quotient map $$p \colon C \times E \rightarrow  (C \times E)/G.$$
Then, since $p$ is a finite \'etale cover,
the cotangent bundle $\Omega_{X}$ is nef and $\kappa(K_X)=\nu(K_X)=1$ holds.
The morphism $f\colon X \rightarrow  C/ \langle  i \rangle \cong \mathbb{P}^1$ induced by the first projection
is the Iitaka fibration of $X$.
Note that the base variety $\mathbb{P}^1$ is, of course, not of general type,
but $f\colon X \rightarrow  C/ \langle  i \rangle \cong \mathbb{P}^1$
is the core fibration of $X$ and a morphism of general type.

\end{ex}

We prove a structure theorem in the case of $\nu(K_{X})=1$.

\begin{thm}
\label{Structure_compact_kahler}
Let $X$ be a compact K\"ahler manifold
such that $\Omega_{X}$ is nef and $\nu(K_{X})=1$.
Then, there exists a finite \'etale cover $X' \to X$ such that
the Iitaka fibration $f' \colon X '\rightarrow Y'$ is a smooth fibration
onto a smooth curve $Y'$ of genus $\ge 2$ whose fibers are complex tori.
\end{thm}

\begin{proof}
Theorem \ref{thm-main2} shows that the Iitaka fibration $f\colon X \rightarrow Y$ is a morphism. 
Further, the Iitaka fibration $f\colon X \rightarrow Y$ is equidimensional from $\dim Y =1$,
and the tangent bundle $T_{X}$ is $f$-numerically flat 
(i.e.\,$T_{X}|_{Z}$ is numerically flat for every  subvariety $Z \subset X$ contracted by $f$). 
Hence, by \cite[Lemma 2.1 (2)]{Hor13},
there exists a finite \'etale cover $\pi \colon X' \rightarrow X$
such that $f'$ is a smooth morphism and all fibers of $f'$ are complex tori,
where $f'$ and $\pi'$ are the morphisms that give the Stein factorization of $f \circ \pi$. 
Note that $f'$ is also the Iitaka fibration.
\begin{equation*}
\xymatrix@C=40pt@R=30pt{
X'\ar@{->}[d]_{\pi} \ar@{->}[r]^{f'}&Y' \ar@{->}[d]^{\pi'}\\ 
X\ar@{->}[r]^{f} &  Y \\
 }
 \end{equation*}

We show that $K_{Y'}$ is big.
By the proof of Theorem \ref{thm-main2}, the variety $X'$ is not special, 
thus the morphism $f'\colon X' \to Y'$ gives the core fibration of $X'$ by \cite[Corollary 3.18]{Cam04}. 
By the definition of the core fibration, the divisor $K_{Y'} + \Delta(f')$ is big. 
In our case, since $f'$ is smooth, we have $\Delta(f')=0$,
which implies that $K_{Y'}$ is big.
\end{proof}

The converse of Theorem \ref{Structure_compact_kahler} also holds, 
as pointed out by Xiaojun Wu.

\begin{thm}
\label{thm-Torus-fibration}
Let $ f\colon X \to C $ be a smooth fibration
onto a smooth curve $C$ of genus $\ge 2$ whose fibers are complex tori. 
Then, the numerical Kodaira dimension \( \nu(K_{X}) \) is equal to one, and the cotangent bundle \( \Omega_{X} \) is nef. In particular, \( c_2(\Omega_{X})=0 \).
\end{thm}

\begin{proof}
Since \( f \) is a smooth torus fibration, the curvature of \( K_{X/C} \) is given by the pullback of the Weil--Petersson form 
(see \cite[Equation 1.1]{BCS20}): 
\[
2\pi c_1(K_{X/C}) \;=\; \frac{1}{\mathrm{vol}(X_c)}\, f^{*}\omega_{\mathrm{WP}}.
\]  
We see that \( c_1(\Omega_X)^2 = 0 \) by \( \dim C = 1 \). 
The relative cotangent bundle \( \Omega_{X/C} \) is  nef by  \cite[Proposition 5]{Kra97}, 
and thus \( \Omega_{X} \)  is nef since \( f^*\Omega_C \) is nef. 
\end{proof}

\section{On Variation of the Fibers of Iitaka Fibrations}\label{sec-fujita}

\subsection {On the Fujita decomposition of nef cotangent bundles}

In this subsection, using the Fujita decomposition,
we study the abundance conjecture
for smooth projective varieties with nef cotangent bundles.
To motivate this study,
we first observe the Fujita decomposition of the cotangent bundle in the case of \cite{WZ02, Liu14}.

\begin{obs}{(Fujita decomposition in the case of nonpositive holomorphic bisectional curvature). }\label{obs1}
Let $X$ be a smooth projective variety admitting
a K\"ahler metric $g$ with nonpositive holomorphic bisectional curvature.
Wu-Zheng and Liu in \cite{WZ02, Liu14} showed that
a holomorphic Ricci kernel foliation $\mathcal{R}$ induced by the metric $g$ 
determines a locally trivial fibration $f\colon X \to Y$,
which actually coincides with the Iitaka fibration.
Then, we can show that the standard exact sequence of tangent bundles
$$
0 \to f^{*}\Omega_{Y} \to \Omega_{X} \to \Omega_{X/Y}=\mathcal{R} \to 0
$$
corresponds to the Fujita decomposition of $\Omega_{X}$.
\end{obs}

Under the weaker assumption that the cotangent bundle $\Omega_{X}$ is nef,
we cannot construct a holomorphic Ricci kernel foliation $\mathcal{R}$
and we do not expect that the Iitaka fibration $f\colon X \to Y$ is locally trivial.
However, we can still consider the flat part of the Fujita decomposition of $\Omega_{X}$.
Therefore, it is natural to ask how the Fujita decomposition of $\Omega_{X}$ is related to the Iitaka fibration of $X$.
Specifically, in this paper, we study whether the Fujita decomposition determines the Iitaka fibration and
what relation exists between the flat part and the local triviality.

We first begin with the definition of the Fujita decomposition.

\begin{prop}[{cf.\,\cite[Section 3 ]{Iwai20}}]\label{fujita}
Let $E$ be a nef vector bundle on a smooth projective variety $X$.
\\
$(a)$ There exists a numerically flat vector bundle $F$
and a generically ample vector bundle $H$ such that
the sequence $0 \to H \to E \to F \to 0 $ is exact.
\\
$(b)$ In the case of $E=\Omega_{X}$, we further have: 
\begin{enumerate}
\item[$(1)$]  $F$ is Hermitian flat, $\det F$ is a torsion line bundle, and $F^{*} \subset T_{X}$ is a regular foliation.
\item[$(2)$] $0 \to H \to E \to F \to 0 $ splits.
\end{enumerate}
\end{prop}
\begin{proof}
Conclusion (a) has been proven in \cite[Lemma 3.1]{Iwai20}.
We prove only Conclusion (b) by reviewing \cite[Lemma 3.1]{Iwai20}.

Let $A$ be an ample line bundle on $X$.
When $\mu_{A^{n-1}}^{\min}(\Omega_{X})>0$,
the cotangent bundle $\Omega_{X}$ is generically ample by Lemma \ref{nef_generically_ample}.
Then, we see that $F:=0$ and $H:=\Omega_{X}$ satisfy the desired conclusions.
Hence, we may assume that $\mu_{A^{n-1}}^{\min}(\Omega_{X})=0$.

In this case, 
we define $\mathcal{F}$ by the maximal destabilizing subsheaf of $T_X $ with respect to $A^{n-1}$.
Then, we have
$$
\mu_{A^{n-1}}(\mathcal{F})=\mu_{A^{n-1}}^{\max}(T_X) = -\mu_{A^{n-1}}^{\min}(\Omega_{X}) =0.
$$
The sheaf $\mathcal{F}$ is pseudoeffective, as is $\det \mathcal{F}^{*}$
since it is a quotient of the nef cotangent bundle.
We obtain $c_1(\mathcal{F})=0 \in H^{1,1}(X, \mathbb{R})$
from $c_1(\mathcal{F}^{*})A^{n-1}=0$.
Hence, the sheaf $\mathcal{F}$ is a numerically flat vector bundle
(see \cite[Theorem 1.20]{GKP16} or \cite[Theorem 1.2]{HIM21}).
On the other hand, the canonical bundle $K_X$ is nef,
and $\mathcal{F}$ is a saturated subsheaf with $c_1(\mathcal{F})=0$. 
Thus, by \cite[Theorem 5.2]{LPT18} and \cite[Lemma 2.1]{PT14}, we see that 
\begin{itemize}
\item[$\bullet$] $\mathcal{F} \subset T_{X}$ is a regular foliation,
\item[$\bullet$] $\mathcal{F}$ is polystable with respect to $A$,
\item[$\bullet$] $K_{\mathcal{F}}=\det \mathcal{F}^{*}$ is a torsion line bundle,
\item[$\bullet$] $T_X =  \mathcal{F}\oplus \mathcal{G}$ holds for some regular foliation $\mathcal{G} \subset T_{X}$.
\end{itemize}
The vector bundle $\mathcal{F}$ is numerically flat and polystable,
and thus $\mathcal{F}$ is Hermitian flat. 
Set $F:=\mathcal{F}^{*}$. 
Then, it remains to show that $G:=\mathcal{G}^{*}$ is generically ample. 
If $\mu_{A^{n-1}}^{\max}(\mathcal{G}) \ge 0$,
we can find a nonzero subsheaf $\mathcal{N} \subset \mathcal{G}$ with $\mu_{A^{n-1}}(\mathcal{N}) \ge 0$.
This contradicts the maximality of $\mathcal{F}$
by $\mu_{A^{n-1}}(\mathcal{N} \oplus \mathcal{F}) \ge 0$.
Hence, we obtain $\mu_{A^{n-1}}^{\max}(\mathcal{G}) <0$,
and thus $\mathcal{G}^{*}$ is generically ample with respect to $A$. 
By Lemma \ref{nef_generically_ample}, we see that $\mathcal{G}^{*}$  is generically ample 
since $\mathcal{G}^{*}$ is nef.
\end{proof}

\begin{ex}
Conclusion (b) in Proposition \ref{fujita} does not hold for a nef vector bundle $E \not =\Omega_{X}$.
A numerically flat vector bundle is not necessarily Hermitian flat (for example, see \cite{DPS94}).
Hence, Property (1) does not hold in general.
When $X$ admits an ample line bundle $L$ with $H^1(X,L) \neq0$,
we obtain a nonsplitting exact sequence: 
$$
0 \rightarrow L \rightarrow E \rightarrow \mathcal{O}_{X} \rightarrow 0
$$
from $H^1(X,L) \cong {\rm Ext}^{1}(\mathcal{O}_{X}, L) \neq 0$.
This sequence gives the Fujita decomposition of $E$ but does not split by construction.
\end{ex}

The Fujita decomposition shows that Conjecture \ref{Abundance_nef}
can be reduced to the case where $\Omega_{X}$ is nef and generically ample.

\begin{conj}$\hspace{-0.2cm}{}_{\le n}$\label{Abundance_nef}
Let $X$ be a smooth projective variety of dimension at most $n$ with nef cotangent bundle.
Then, the canonical bundle $K_X$ is semi-ample.
\end{conj}

\begin{prop}
Let $X$ be a smooth projective variety of dimension $n$ with  nef cotangent bundle.
If the flat part of the Fujita decomposition of $\Omega_{X}$ is nonzero and Conjecture \ref{Abundance_nef}${}_{\le n-1}$ holds, then $K_X$ is semi-ample. 
\end{prop}
\begin{proof}
Let $\Omega_{X}=F \oplus H$ be the Fujita decomposition as in Lemma \ref{fujita}.
Since $F$ is a nonzero Hermitian flat vector bundle by assumption,
there exists a representation
$$
\rho\colon \pi_1(X) \rightarrow U(r, \C)
$$
such that $F \cong X_{\univ} \times_{\rho} \C^{r}$,
where $r := {\rm rk} F >0$ and $X_{\univ}$ is the universal cover of $X$.

In the case where $\Image{\rho}$ is finite,
we may assume that $F $ is a trivial vector bundle
after replacing $X$ with a finite \'etale cover.
Then, we have $H^{0}(X, T_{X}) \neq 0$,
and thus, up to finite \'etale covers,
we can find an abelian variety $A$ and a smooth projective variety $Y$ with $H^0 (Y,T_Y)=0$
such that $X \cong A \times Y$ by a theorem of Lieberman \cite{Lie78} (cf.\,\cite[Theorem 3.2]{Amoros}).
The canonical bundle $K_Y$ is semi-ample by Conjecture \ref{Abundance_nef},${}_{\le n-1}$,
and thus so is $K_X$.

We may assume that $\Image{\rho}$ is infinite.
Let $G$ be the algebraic Zariski closure of $\Image{\rho}$.
Up to finite \'etale covers, we may assume that $G$ is connected.
Let $K$ be the kernel of $\bar \rho$ defined by
$$ \bar \rho\colon \pi_1(X) \xrightarrow{\quad \rho \quad} G \twoheadrightarrow G / \Rad(G),$$
where $\Rad(G)$ is the solvable radical of $G$.
By Selberg's Lemma, we may assume that the image of $\bar \rho$ is torsion-free.

We consider the $K$-Shafarevich map $$\sha^{K}\colon X \dashrightarrow \Sha^{K}(X).$$
Since $\Image{\rho}$ is not virtually solvable by \cite[Corollary 2.7]{PT14},
we obtain $\dim \Sha^{K}(X)>0$.
Since $G / \Rad(G)$ is semi-simple and the image of $\bar \rho$ is torsion-free,
the base variety $\Sha^{K}(X)$ is of general type
by \cite[Th\'eor\`eme 1]{CCE15}.
The canonical bundle $K_X$ is abundant
by Conjecture \ref{Abundance_nef}${}_{\le n-1}$ and \cite[Proposition 2.10]{GKP20b}.
\end{proof}

\subsection{Hermitian flatness of nef cotangent bundles}\label{section:flat}

In this subsection, assuming the abundance conjecture,
we study the associated Iitaka fibration $f\colon X\to Y$ of
a smooth projective variety $X$ with the nef cotangent bundle $\Omega_{X}$.
In particular, we study a relation among
the Fujita decomposition of $\Omega_{X}$,
a certain flatness of $\Omega_{X/Y}$, and
a variation of the fibers.

Compared to Observation \ref{obs1},
when $\Omega_{X}$ is merely nef,
the Iitaka fibration $f\colon X \to Y$ is not necessarily locally trivial.
This difference can be captured by the moduli of the abelian varieties.
Precisely, assuming the abundance conjecture,
H\"oring in \cite{Hor13} proved that, up to finite \'etale covers of $X$, 
the Iitaka fibration $f\colon X \to Y$ is an abelian group scheme, i.e.,
there exists a fibration $g\colon Y \to Z$ onto a subvariety $Z$
in the fine moduli of polarized abelian varieties with level structure
so that $f\colon X \to Y$ is recovered from the fiber product via $g\colon Y \to Z$.
Furthermore, the canonical bundle $K_{Y}$ of $Y$ is ample. 
From this viewpoint, the local triviality of $f\colon X \to Y$ can be rephrased as $Z$ being a one point.
In this subsection,
we prove that the dimension $\dim Z$,
which measures how far $f\colon X \to Y$ is from the local triviality of $f$,
is related to a certain flatness.

In this section, we suppose that $f\colon X \to Y$ is an abelian group scheme. 
Then, keeping the above situation,
we prove Theorems \ref{thm-main3} and \ref{thm-flat}.

\smallskip

We consider the Fujita decomposition of the nef cotangent bundle $\Omega_{X}$: 
\begin{align}\label{eq-fuj}
0 \to \mathcal{H} \to \Omega_{X} \to \mathcal{F} \to 0, 
\end{align}
where $\mathcal{F}$ is a Hermitian flat vector bundle and
$\mathcal{H}$ is a generically ample vector bundle.
The relative cotangent bundle $\Omega_{X/Y}$ is nef
by the standard exact sequence of tangent bundles
\begin{align}\label{eq-sta}
0 \to f^{*}\Omega_{Y} \to \Omega_{X}  \to \Omega_{X/Y}  \to 0. 
\end{align}
Hence, we can consider the Fujita decomposition of $\Omega_{X/Y}$: 
\begin{align}\label{eq-fuj2}
0 \to \mathcal{G} \to \Omega_{X/Y} \to \mathcal{E} \to 0, 
\end{align}
where $\mathcal{E}$ is a numerically flat vector bundle and
$\mathcal{G}$ is a generically ample vector bundle.
Then, we have: 

\begin{prop}\label{prop:1}
The flat part of $\mathcal{F}$ of $\Omega_{X}$ coincides
with the flat part $\mathcal{E}$ of $\Omega_{X/Y}$.
\end{prop}
\begin{proof}

We first prove that $f^{*}\Omega_{Y}$ is a generically ample vector bundle on $X$.
For this purpose, we confirm that $\Omega_{Y}$ (and thus $f^{*}\Omega_{Y}$) is a nef vector bundle.
There exists a section $s\colon Y \to X$ of $f\colon X \to Y$ since it is an abelian group scheme,
and thus
we obtain a surjective bundle morphism $\Omega_{X}|_{s(Y)} \to \Omega_{s(Y)}$ on the image $s(Y) \subset X$.
This implies that  $\Omega_{Y}$ is nef
since $\Omega_{s(Y)} $ can be identified with $\Omega_{Y}$
via the isomorphism $s\colon Y \cong s(Y)$.

We now show that $f^{*}\Omega_{Y}$ is a generically ample vector bundle.
The canonical bundle $K_{Y}$ is ample by \cite{Hor13};
hence, $Y$ admits a K\"ahler-Einstein metric.
This implies that $\Omega_{Y}$ is semistable.
Let us consider a curve 
$$C:=A_{1} \cap A_{2} \cap \cdots \cap A_{\dim Y -1}$$
defined by general members $A_{i} \in |m K_{Y}|$, where $m \gg 1$.
Then, by Mehta-Ranamathan's theorem,
the restriction $\Omega_{Y}|_{C}$ is also semistable.
Miyaoka's characterization of semistabilities
shows that the $\mathbb Q$-vector bundle
$$
\Omega_{Y}|_{C}  \otimes \det (\Omega_{Y}|_{C})^{-\frac{1}{\dim Y} }
$$
is nef.
This implies that $\Omega_{Y}|_{C}$ is ample since $K_{Y}$ is ample.
We take a curve $C' \subset X$ such that the restriction $f|_{C'}\colon C' \to C$ is a surjective morphism.
Then, we see that $$(f^{*} \Omega_{Y})|_{C'}=(f|_{C'})^{*} (\Omega_{Y}|_{C})$$ is an ample vector bundle on $C'$
since $f|_{C'}\colon C' \to C$ is a finite morphism.
However, for the flat part $\mathcal{F}'$ of the nef vector bundle $f^{*}\Omega_{Y}$,
we have the surjective morphism $(f^{*} \Omega_{Y} )|_{C'} \to \mathcal{F'}|_{C'}$
by restricting $f^{*}\Omega_{Y} \to \mathcal{F'}$ to $C'$.
The vector bundle $(f^{*} \Omega_{Y} )|_{C'}$ is ample by the choice of $C'$,
but the restriction $\mathcal{F'}|_{C}$ is still flat.
This is a contradiction.
Hence, we conclude that $\mathcal{F}' =0$ (i.e., $f^{*}\Omega_{Y}$ is generically ample).

By using exact sequences \eqref{eq-sta} and \eqref{eq-fuj2},
the snake lemma yields the following commutative diagram.

\begin{align}\label{eq-comm}
\xymatrix{
&&0&0&&\\
0 \ar[r]& \mathcal{G^{*}} \ar[r] & T_{X}/\mathcal{E}^{*} \ar[r] \ar[u] & f^{*} T_{Y}\ar[r] \ar[u]&0 &\\
 & & T_{X} \ar[u] \ar@{=}[r]&T_{X}  \ar[u]&&\\
&0 \ar[r]&\mathcal{E}^{*} \ar[u] \ar[r]& T_{X/Y} \ar[u] \ar[r]&\mathcal{G}^{*}\ar[r]& 0\\
&&0\ar[u]&0.\ar[u]&&
}
\end{align}
The dual bundle $(T_{X}/\mathcal{E}^{*})^{*}$ is generically ample
since $\mathcal{G}$ and $f^{*} \Omega_{Y}$ are as well.
This implies that the (dual of) the left column in the above diagram
actually corresponds to the Fujita decomposition of $\Omega_{X}$.
Hence, the flat part $\mathcal{F}$ of $\Omega_{X}$ is
the flat part $\mathcal{E}$ of $\Omega_{X/Y}$.
\end{proof}

Theorem \ref{thm-main3} immediately follows from Proposition \ref{prop:1}.

\begin{proof}[Proof of Theorem \ref{thm-main3}]
The moduli part $M$ of $f\colon X \to Y$ defined by $f^{*}M:=K_{X/Y}$ is
semi-ample by \cite[Theorem 1.2]{Fuj} and \cite{Ue78}.
The semi-ample fibration $h\colon Y \to W$ of $M$ coincides with $g\colon Y \to Z$,
and thus $M=g^{*}A$ holds for some ample line bundle $A$ on $Z$.
Indeed, for a point $z \in Z$,
the restriction 
$$f|_{X_{z}}\colon X_{z}:= (g\circ f)^{-1}(z) \to Y_{z}:=g^{-1}(z)$$
is a locally trivial fibration, and thus $M|_{Y_{z}}$ is trivial.
Conversely, for a point $w \in W$,
the relative canonical bundle $K_{X_{w}/Y_{w}}$
is trivial, and thus 
$$f|_{X_{w}}\colon X_{w}:= (h\circ f)^{-1}(w) \to Y_{w}:=g^{-1}(w)$$ is locally trivial 
(see \cite[Theorem 4.7]{Amb05} or \cite[Theorem 1.6]{EIM21}). 
This implies that $g (Y_{w})$ comprises one point.

If $f\colon X \to Y$ is locally trivial (i.e., $\dim Z =0$),
we have $c_{1}(\Omega_{X/Y})=0$ by $$K_{X/Y}=f^{*}M=f^{*}g^{*}A.$$
Hence, since $\Omega_{X/Y}$ is nef, we see that $\Omega_{X/Y}$ is numerically flat, and thus
the vector bundle $\Omega_{X/Y}$ coincides with the flat part of $\Omega_{X}$,
which is Hermitian flat.
Hence, Condition (3) leads to Conditions (1) and (2).

It is obvious that Condition (1) is equivalent to Condition (2).
Therefore, it remains to show that Condition (1) leads to Condition (3).
The Fujita decomposition of cotangent bundles always splits,
and the sequence \eqref{eq-sta} of tangent bundles also splits
under Condition (1).
This implies that the Kodaira-Spencer map vanishes everywhere.
Hence, the Iitaka fibration $f\colon X \to Y$ is locally trivial.
\end{proof}

We finally prove Theorem \ref{thm-flat}.
\begin{proof}[Proof of Theorem \ref{thm-flat}]

By pushing forward \eqref{eq-comm} by $f$,
we obtain: 
\begin{align}
\label{eq-comm2}\xymatrix{
&&R^{1}f_{*} \mathcal{E}^{*} \ar[r]&R^{1}f_{*} T_{X/Y}&&\\
0 \ar[r]& f_{*}\mathcal{G^{*}} \ar[r] & f_{*}\mathcal{H}^{*}=f_{*}(T_{X}/\mathcal{E}^{*})
 \ar[r]^{\qquad \quad \psi} \ar[u] &  T_{Y}\ar[r] \ar[u]_{\varphi}&
 \cdots  &\\
 & & f_{*}T_{X} \ar[u] \ar@{=}[r]&f_{*}T_{X}  \ar[u]&&\\
&0 \ar[r]&f_{*} \mathcal{E}^{*} \ar[u] \ar[r]&f_{*} T_{X/Y} \ar[u] \ar[r]&  \cdots & \\
&&0\ar[u]&0,\ar[u]&&
}
\end{align}
where $\psi$ and $\varphi$ are the induced morphisms.
Note that $\varphi$ is the Kodaira-Spencer map.
The Fujita decomposition \eqref{eq-fuj} splits,
and thus the induced morphism $f_{*}\mathcal{H}^{*} \to R^{1}f_{*} \mathcal{E}^{*} $
is the zero map.
Hence, for a general point $y \in Y$,
we obtain
$$\Ker{\varphi_{y}} =  \Image{\psi_{y}} \cong (f_{*}\mathcal{H}^{*})_{y}/ (f_{*}\mathcal{G}^{*})_{y}
=H^{0}(X_{y}, \mathcal{H}^{*}|_{X_{y}})/H^{0}(X_{y}, \mathcal{G}^{*}|_{X_{y}}).$$
Furthermore, we have $T_{Y/Z, y} = \Ker{\varphi_{y}}$
since $g\colon Y \to Z$ is the morphism into the moduli of the abelian varieties
as in the proof of Theorem \ref{thm-main3}.
Hence, we obtain
\begin{align*}
\dim Y - \dim Z &= h^{0}(X_{y}, \mathcal{H}^{*}|_{X_{y}}) - h^{0}(X_{y}, \mathcal{G}^{*}|_{X_{y}})
\end{align*}
We now have the exact sequence
$$0 \to  \mathcal{E}^{*}|_{X_{y}} \to T_{X/Y}|_{X_{y}}  \to \mathcal{G}^{*}|_{X_{y}} \to 0$$
from \eqref{eq-fuj}.
The vector bundle $\mathcal{E}^{*}|_{X_{y}}$ is numerically flat, and $T_{X/Y}|_{X_{y}}$ is a trivial vector bundle.
This implies that this sequence splits, 
and thus that $\mathcal{G}^{*}|_{X_{y}}$ is numerically flat and globally generated.
Hence, we see that $\mathcal{G}^{*}|_{X_{y}}$ is a trivial vector bundle by \cite[Proposition 1.16]{DPS94}.
This implies that 
$$h^{0}(X_{y}, \mathcal{G}^{*}|_{X_{y}}) = \dim X - \dim Y - \rk \mathcal{E},$$
and thus, we conclude that
\begin{align}\label{key}
\dim Z = \rk \mathcal{H} - h^{0}(X_{y}, \mathcal{H}^{*}|_{X_{y}}). 
\end{align}

For the proof of Theorem \ref{thm-flat}, we investigate $h^{0}(X_{y}, \mathcal{H}^{*}|_{X_{y}})$ in detail. 
Let us consider the induced commutative diagram: 
\begin{align}\label{eq-comm3}
\xymatrix{
&0&0&&&\\
0 \ar[r]& \mathcal{G}|_{X_{y}} \cong T_{\rk \mathcal G} \ar[u]\ar[r] & \Omega_{X/Y}|_{X_{y}} \cong T_{n-m} \ar[r] \ar[u] 
& \mathcal{E}|_{X_{y}} \cong T_{\rk \mathcal{E}}
\ar[r] &0 &\\
0 \ar[r] & \mathcal{H}|_{X_{y}} \ar[u]_{\alpha} \ar[r]_{\beta} & \Omega_{X}|_{X_{y}} \ar_{p}[u] \ar[r]& \mathcal{E}|_{X_{y}} \cong T_{\rk \mathcal{E}}\ar@{=}[u]\ar[r]
&0&\\
0\ar[r]&\Ker{\alpha}\cong \beta^{-1}(f^{*}\Omega_{Y}|_{X_{y}})\ar[r]\ar[u]  &f^{*}\Omega_{Y}|_{X_{y}} \cong T_{m} \ar_{i}[u] && \\
&0\ar[u]&0,\ar[u]&&&
}
\end{align}
where $n=\dim X$ and $m:= \dim Y$.
Here, the notation $T_{k}$ denotes the trivial vector bundle on $X_{y}$ of rank $k$.
Note that $\Omega_{X}|_{X_{y}} $ is numerically flat from the middle column.

The following lemma will be applied to $H:=\mathcal{H}|_{X_{y}}$. 
Note that $H:=\mathcal{H}|_{X_y}$ satisfies the assumptions of the lemma.
Indeed, the left vertical exact sequence in the above diagram shows that
$$
\Ker{\alpha}
=
\beta^{-1}\bigl(f^*\Omega_Y|_{X_y}\bigr).
$$
is numerically flat since both $\mathcal{H}|_{X_y}$ and
$\mathcal{G}|_{X_y}$ are numerically flat.
Moreover, since the bottom exact sequence splits and $f^*\Omega_Y|_{X_y}$ is a trivial vector bundle,
it is globally generated, and thus also a trivial vector bundle.

\begin{lem}\label{lem:1}
Let $A$ be a vector bundle on a compact complex manifold.
Assume that $A$ is an extension of a trivial vector bundle by a trivial vector bundle.
Then, any Hermitian flat subbundle $F \subset A$ is a trivial vector bundle.
\end{lem}
\begin{proof}
By assumption, we have an exact sequence
\[
0 \longrightarrow T_a \xrightarrow{\ i\ } A
\xrightarrow{\ p\ } T_b \longrightarrow 0,
\]
where $T_k$ denotes the trivial vector bundle of rank $k$.
This implies that $A$ is numerically flat.
The above exact sequence and the inclusion $F\subset A$ give
the following commutative diagram:
\begin{align*}
\xymatrix{
&0&0&0&&\\
0 \ar[r]& p(F) \ar[u]\ar[r] & T_b \ar[r] \ar[u]
& T_b/p(F)\ar[r] \ar[u] &0 &\\
0 \ar[r] & F \ar[u] \ar[r] & A \ar_{p}[u] \ar[r]
& A/F \ar[u] \ar[r] &0&\\
0 \ar[r]& i^{-1}(F) \ar[u] \ar[r]
&T_a \ar_{i}[u] \ar[r] &T_a/i^{-1}(F) \ar[u] \ar[r]&0\\
&0\ar[u]&0\ar[u]&0.\ar[u]&&
}
\end{align*}

We first verify that the inclusions $p(F) \subset T_{b}$, 
$i^{-1}(F) \subset T_{a}$, and $i^{-1}(F) \subset F$ are subbundles.
Consider the sheaf morphism
\[
p_F:=p|_F\colon F\longrightarrow T_b
\]
between the Hermitian flat vector bundles $F$ and $T_b$.
The morphism $p_F$ can be regarded as a section $s$ of
the Hermitian flat vector bundle $\Hom(F,T_b)$.
The squared norm $|s|_h^2$ is plurisubharmonic and hence constant by
the compactness of $X$, where $h$ is a Hermitian flat metric on
$\Hom(F,T_b)$. The constancy of
$|s|_h^2$ shows that the components of $s$ in a local parallel unitary frame are constant.
This implies that $p_F$ is parallel and is a bundle morphism. 
Consequently, we can conclude that 
$\ker(p_F) \subset F$ and $p(F)=\Image{p_F} \subset T_{b}$ are subbundles. 
By the exactness of the sequence, 
the injective bundle morphism $i$ induces an isomorphism of vector
bundles
\[
i\colon T_a\xrightarrow{\ \cong\ }\ker (p).
\]
Moreover, we have 
\[
i\bigl(i^{-1}(F)\bigr)
=
F\cap i(T_a)
=
F\cap\ker p
=
\ker(p_F).
\]
Since $\ker(p_F) \subset \ker (p)=i(T_a)$ is a subbundle, its inverse
image under the bundle isomorphism $i\colon T_a\to i(T_a)$ is a 
subbundle of $T_a$. Hence, we see that $i^{-1}(F)\subset T_a$ is a 
subbundle. Furthermore, the restriction
\[
i|_{i^{-1}(F)}\colon
i^{-1}(F)\xrightarrow{\ \cong\ }\ker(p_F)
\]
is an isomorphism of vector bundles. Thus, via $i$, the vector bundle
$i^{-1}(F)$ is identified with the vector subbundle
$\ker(p_F)\subset F$.

We show that $T_a/i^{-1}(F)$ and $T_b/p(F)$ are trivial.
Both $T_a/i^{-1}(F)$ and $T_b/p(F)$ are quotients of trivial vector bundles, 
and thus they are globally generated and nef.
Since a globally generated numerically flat vector bundle is trivial,
it suffices to show that their first Chern classes vanish.
Considering the middle row, the quotient bundle $A/F$ is nef and
\[
c_1(A/F)=c_1(A)-c_1(F)=0.
\]
This shows that $A/F$ is numerically flat. By the right column, we have
\[
0=c_1(A/F)
=
c_1\bigl(T_a/i^{-1}(F)\bigr)
+
c_1\bigl(T_b/p(F)\bigr).
\]
Both terms on the right-hand side are nef and therefore vanish.
Then, we obtain the desired triviality.

The top and bottom rows split. 
These splittings follow since 
 $T_b$ and $T_a$ are Hermitian
flat and the quotients $T_b/p(F)$ and
$T_a/i^{-1}(F)$ are numerically flat. 
Consequently, we see that $p(F)$ and $i^{-1}(F)$ are trivial vector bundles.
Applying the same argument to the left column, we obtain a splitting
\[
F\cong i^{-1}(F)\oplus p(F).
\]
Since both summands are trivial, the vector bundle $F$ is  trivial.
\end{proof}

We confirm Theorem \ref{thm-flat} (1); i.e., $\dim Z = \rk H - l$ holds.
Recall that the invariant $l$ in Theorem~\ref{thm-flat} is defined as the
maximal rank of Hermitian flat subbundles of $H$.
We claim that
$
h^0(X_y,H)=l.
$
Indeed, let $F\subset H$ be a Hermitian flat subbundle of maximal rank
$l$. By Lemma~\ref{lem:1}, the vector bundle $F$ is trivial. Hence, we have 
$$
l=\rank F
=h^0(X_y,F)
\leq h^0(X_y,H).
$$
Conversely, Lemma~\ref{lem:1} shows that 
the sections of $H^0(X_y,H)$ determines a trivial (and thus Hermitian flat) subbundle of $H$. 
Therefore, we obtain 
$
l\geq h^0(X_y,H),
$
and the claim follows.
By Lemma~\ref{lem:1} again, 
we see that 
$$h^{0}(X_{y}, \mathcal{H}^{*}|_{X_{y}})
=h^{0}(X_{y}, \mathcal{H}|_{X_{y}})=l.
$$
Therefore, Theorem \ref{thm-flat} (1) follows from \eqref{key}. 

\begin{rem}\label{rem-fil}
As in \cite[Theorem 1.18]{DPS94}, 
we take a filtration of the numerically flat vector bundle $H:=\mathcal{H}|_{X_{y}}$ by subbundles $\{F_k\}_{k=0}^{p}$, i.e.,
$$0=\colon F_0 \subset F_1 \subset \dots \subset F_{p-1} \subset F_p:=H,$$
such that each quotient bundle $F_k/F_{k-1}$ is a Hermitian flat vector bundle. 
Then, we see that $F_k/F_{k-1}$ is a trivial line bundle and $p=\rk H$. 
Although this fact is not needed in the proof,  we include the verification for completeness. 

By construction,
the vector bundle $F_{k}$ is inductively obtained as follows: 
$F_{k}$ is the inverse image of $\overline{F}_{k}$
by $H \to H/F_{k-1}$,
where $\overline{F}_{k}$ is a subbundle of $H/F_{k-1}$ of the minimal rank
that satisfies $c_{1}(\overline{F}_{k}) \cdot c_{1}(A)^{\dim X -1} \geq 0$.
The vector bundle $H=\mathcal{H}|_{X_{y}}$ satisfies the assumption of Lemma \ref{lem:1},
and $F_{1}$ is Hermitian flat, and thus $F_{1}$ is a trivial vector bundle.
This implies that $F_{1}$ is a trivial line bundle since $F_{1}$ has the minimal rank.
The quotient $H/F_{1}$
also satisfies the assumption of Lemma \ref{lem:1}
by the same argument as in Lemma \ref{lem:1}.
Hence, it follows that $\overline{F}_{2}$ is a trivial line bundle
since $\overline{F}_{2}$ is Hermitian flat and has the minimal rank.
Therefore, we can conclude that $p=\rk H$ and
$F_k/F_{k-1}$ is a trivial line bundle.
\end{rem}


We finally confirm Theorem \ref{thm-flat} (2).
Let $p\colon \mathbb{P}(\Omega_{X}) \to X$ be the projective space bundle over $X$.
Then, we have
$$
p_{*}\big( K_{\mathbb{P}(\Omega_{X}) /X} +
\mathcal{O}_{\mathbb{P}(\Omega_{X})}(n+1) \big)
=\Omega_{X} \otimes \det \Omega_{X}=\Omega_{X}\otimes K_{X}.
$$
By assumption, the hyperplane bundle $\mathcal{O}_{\mathbb{P}(\Omega_{X})}(1)$ is semipositive,
and thus, we can conclude that $\Omega_{X} \otimes K_{X}$ is Griffiths semipositive
by applying the theory of positivity of direct images (see \cite{PT, HPS, BP08}).
The restriction $K_{X}|_{X_{y}}$ is numerically trivial since $f\colon X \to Y$ is the Iitaka fibration.
This implies that $\Omega_{X}|_{X_{y}}$ is Griffiths semipositive. 
By the restriction of \eqref{eq-fuj} to $X_{y}$, 
we deduce that $\mathcal{H}|_{X_{y}}$ is also Griffiths semipositive. 
Any Griffiths semipositive vector bundle with vanishing first Chern class is Hermitian flat, 
and thus  $H=\mathcal{H}|_{X_{y}}$ is Hermitian flat. 
Then, Theorem \ref{thm-flat} (2) follows from (1).
\end{proof}

\section{Generically Nef Vector Bundles with Vanishing Second Chern Class}
\label{Sec5}

This section describes the proof of Theorem \ref{generically_nef_reflexive}.

\subsection{Case of smooth projective varieties}
\label{Chapter_projective}

We first begin with the following lemma.
This lemma follows from the proof of \cite[Proposition 4.6]{Nakayama} or \cite[Proposition 13]{Wu20},
but we give a proof for reader's convenience.
\begin{lem}\label{Technical_Lemma}
Let $X$ be a complex manifold and let
$$0 \rightarrow \mathcal{F} \rightarrow \mathcal{E} \rightarrow \mathcal{G} \rightarrow 0$$
be an exact sequence of torsion-free sheaves on $X$.
Assume that $\mathcal{F}$ is locally free, $ \mathcal{E}$ is reflexive, $\mathcal{G}^{**}$ is locally free, and $\codim (\Supp(\mathcal{G}^{**} / \mathcal{G})) \ge 3$.
Then, both $\mathcal{E}$ and $\mathcal{G}$ are locally free on $X$. 
\end{lem}

\begin{proof}
Take a Zariski open set $X_0 \subset X$ with $\codim(X \setminus X_0) \ge 3$ such that
$\mathcal{G} |_{X_0} \cong \mathcal{G}^{**} |_{X_0} $ and $\mathcal{E}|_{X_0}$ are locally free.
Let us consider the following exact sequence: 
\begin{equation}
\label{exact_sequence}
0 \rightarrow \mathcal{F} |_{X_0} \rightarrow \mathcal{E} |_{X_0} \rightarrow \mathcal{G}^{**} |_{X_0} \rightarrow 0.
\end{equation}
Since $\mathcal{G}^{**} |_{X_0} $ is locally free,
we have ${\rm Tor}^{1}_{\mathcal{O}_{X,x}}(\mathcal{O}_{X,x}/\mathfrak{m}_{x}, \mathcal{G}^{**}_{x})=0$ for any $x \in X_0$,
where $\mathfrak{m}_{x}$ is the maximal ideal of $\mathcal{O}_{X,x}$.
Hence, the sequence (\ref{exact_sequence}) is exact as a sequence of vector bundles,
and thus $\mathcal{E} |_{X_0}$ is determined by an extension class in
$H^1(X_0, \mathcal{F} |_{X_0} \otimes (\mathcal{G}^{**} |_{X_0}) ^{*} )$.
By $\codim(X \setminus X_0) \ge 3$ and \cite[Lemma 4]{Wu20}, we have
$$
H^1(X_0, \mathcal{F} |_{X_0} \otimes (\mathcal{G}^{**} |_{X_0}) ^{*} )
\cong
H^1(X, \mathcal{F}  \otimes \mathcal{G}^{*} ).
$$
This implies that $\mathcal{E} |_{X_0}$ can be extended to a vector bundle $\mathcal{E}'$ on $X$.
We obtain $\mathcal{E} = \mathcal{E}'$ by reflexivity.
\end{proof}

We prove Theorem \ref{generically_nef_reflexive} (2)
by using Theorems \ref{Miyaoka_second_chern}, \ref{Ou_chapter6}.

\begin{thm}[{\cite[Theorem 6.1]{Miyaoka87}}]
\label{Miyaoka_second_chern}
Let $X$ be a smooth projective variety of dimension $n \ge 2$, let $\mathcal{E}$ be
a torsion-free coherent sheaf, and let $A_1, \ldots, A_{n-2}$ be ample line bundles on $X$.
Assume that $c_1(\mathcal{E})$ is nef and
$\mu_{A_1\cdots A_{n-2}L}^{\min}(\mathcal{E}) \ge 0$ holds
for any nef line bundle $L$.
Then, we have
$$
c_2(\mathcal{E})A_1 \cdots A_{n-2} \ge 0.
$$
\end{thm}

\begin{thm}[{\cite[Chapter 6]{Ou17}}]
\label{Ou_chapter6}
Let $X$ be a smooth projective variety of dimension $n \ge 2$, let $\mathcal{E}$ be a torsion-free sheaf,
and let $A$ be an ample line bundle.
Assume that $\mathcal{E}$ is reflexive, $c_1(\mathcal{E})$ is nef, $\mathcal{E}$ is generically nef, and $c_2(\mathcal{E})A^{n-2}  = 0$.
Then, we have: 

\begin{itemize}
\item[(1)] In the case of $\nu(c_1(\mathcal{E}))\ge 2$,
there exists a line bundle $\mathcal{E}_1 \subset \mathcal{E}$ such that $\mathcal{E}/\mathcal{E}_{1}$ is torsion-free,
$c_1(\mathcal{E}_1)=c_1(\mathcal{E})$, $c_1(\mathcal{E}/\mathcal{E}_1)=0$, and $c_2(\mathcal{E}/\mathcal{E}_1) A^{n-2}=0$.

\item[(2)] In the case of $\nu(c_1(\mathcal{E})) =1$,
there exists a filtration of subsheaves: 
$$
0 =\colon \mathcal{E}_0 \subset \mathcal{E}_1 \subset \cdots \subset \mathcal{E}_l :=\mathcal{E}
$$
satisfying the following conditions: 
\begin{itemize}
\renewcommand{\theenumi}{\alph{enumi}}
\item[$\bullet$] The above filtration is the $(c_1(\mathcal{E}) +  \varepsilon  A)^{n-1}$-Harder-Narasimhan filtration for any $0< \varepsilon  \ll 1$.
\item[$\bullet$] $c_1(\mathcal{E}_k /\mathcal{E}_{k-1})$ is numerically proportional to $c_1(\mathcal{E})$ for any $k = 1, \dots, l$.
\item[$\bullet$] $c_2(\mathcal{E}_k /\mathcal{E}_{k-1})(c_1(\mathcal{E}) +  \varepsilon  A)^{n-2}=0$
for any $k = 1, \dots, l$ and any $0< \varepsilon  \ll 1$.
\end{itemize}

\item[(3)] In the case of $\nu(c_1(\mathcal{E})) =0$,
the sheaf $\mathcal{E}$ 
is a numerically flat vector bundle.

\end{itemize}
\end{thm}
The assumption that $\mathcal{E}$ is reflexive is necessary since, in the proof, the maximal destabilizing subsheaf has to be reflexive. In particular, this ensures that $\mathcal{E}_1$ appearing in~\textup{(1)} is a line bundle.

\begin{rem}
We should pay attention to the choice of polarizations
when considering generically nef sheaves in the projective setting, as in Theorem \ref{Ou_chapter6},
and in the compact K\"ahler setting, as in Theorem \ref{Cao13_revisit}.
When $X$ is projective, as in Theorems \ref{Miyaoka_second_chern} and \ref{Ou_chapter6},
we may assume that $\mathcal{E}$ is generically nef with respect to every ample polarization, and hence obtain
$
c_2(\mathcal{E})A_1 \cdots A_{n-2} \ge 0
$
for any ample divisors $A_1, \ldots, A_{n-2}$.
However, the corresponding statement is not known in the compact K\"ahler case, since Conjecture \ref{Miyaoka_thm}~\textup{(2)} in Subsection \ref{subsec-rela} remains open. 
For this reason, 
in Theorem \ref{Cao13_revisit}, we only assume that $\mathcal{E}$ is $\{ \eta \}^{n-1}$-generically nef 
for every K\"ahler form $\eta$, and we consider only polarizations of the form $(c_{1}(\mathcal{E}) + \varepsilon {\omega})^{n-1}$ in Theorems \ref{Cao13_revisit} and \ref{generically_nef_reflexive}~\textup{(1)}.
\end{rem}

\begin{proof}[Proof of Theorem \ref{generically_nef_reflexive} $(2)$]
We first remark that $F$ is nef 
if a vector bundle $F$ is projectively numerically flat and $c_{1}(F)$ is nef. 
This because $F \langle \frac{\det F^{*}}{\rk F}\rangle $ is nef. 
The proof is divided by the cases according to $\nu(c_1(\mathcal{E}))$.
\\
\textit{Case of $\nu(c_1(\mathcal{E}))\ge 2$.}
By Theorem \ref{Ou_chapter6}, there exists a nef line bundle $\mathcal{E}_1$ such that
$\mathcal{E}_2 :=\mathcal{E}/\mathcal{E}_1$ is torsion-free, $c_1(\mathcal{E}_2)=0$, and $c_2(\mathcal{E}_2) A^{n-2}=0$. %
Then, from Theorem \ref{Miyaoka_second_chern},
we obtain  $c_2(\mathcal{E}_{2}^{**})A^{n-2}\ge 0$
since $\mathcal{E}_{2}^{**}$ is generically nef.
This implies that
$$
0 = c_2(\mathcal{E}_{2}) A^{n-2} \ge c_2(\mathcal{E}_{2}^{**})A^{n-2}\ge 0,
$$
and thus $c_2(\mathcal{E}_{2}^{**})A^{n-2}= 0$.
Hence, we see that $\mathcal{E}_{2}^{**}$ is
a projectively numerically flat vector bundle by
Theorem-Definition \ref{numerically_projectively_flat}. 
Thus, the above remark shows that $\mathcal{E}_{2}^{**}$ is nef. 
By \cite[Remark 8]{Cao13} or \cite[Lemma 6.9]{Ou17}, we have $\codim (\Supp(\mathcal{E}_{2}^{**}/\mathcal{E}_2)) \ge 3$.
By Lemma \ref{Technical_Lemma}, the sheaf $\mathcal{E}$ is locally free and satisfies the exact sequence
$$
0 \rightarrow  \mathcal{E}_1 \rightarrow  \mathcal{E} \rightarrow \mathcal{E}_{2}^{**} \rightarrow 0.
$$
Therefore, the vector bundle $\mathcal{E}$ is nef since both $\mathcal{E}_1$ and $\mathcal{E}_2^{**}$ are nef. 
This implies that
$$
c_2(\mathcal{E}) = c_1(\mathcal{E}_1)c_1(\mathcal{E}_{2}^{**} ) + c_2(\mathcal{E}_{2}^{**} ) + c_2(\mathcal{E}_{1} ) =0.
$$
\\
\textit{Case of $\nu(c_1(\mathcal{E}))= 1$.}
We use induction on the rank of $\mathcal{E}$.
We show that the subsheaf $\mathcal{E}_1 \subset \mathcal E$
appearing in Theorem \ref{Ou_chapter6} is actually a nef vector bundle.
Note that $\mathcal{E}_1$ is reflexive since $\mathcal{E}_1$ is saturated in $\mathcal{E}$.
The sheaf $\mathcal{E}_1$ is $(c_1(\mathcal{E}) +  \varepsilon  A)^{n-1}$-semistable
and satisfies 
$c_2(\mathcal{E}_1)(c_1(\mathcal{E}) +  \varepsilon  A)^{n-2}=0$.
Furthermore, we have $c_1(\mathcal{E}_1)^2=0$
since $c_1(\mathcal{E}_1)$ is numerically proportional to $c_1(\mathcal{E})$.
Hence, the sheaf $\mathcal{E}_1$ is a projectively numerically flat vector bundle and satisfies
$$c_2(\mathcal{E}_1)=\frac{\rk \mathcal{E}_1 -1}{2\rk \mathcal{E}_1}c_1(\mathcal{E}_1)^2=0.$$
We conclude that $\mathcal{E}_1$ is a nef vector bundle.

The quotient sheaf $\mathcal{Q} :=\mathcal{E}/\mathcal{E}_1$ is a generically nef sheaf
that satisfies that $c_1(\mathcal{Q})$ is nef,
$c_1(\mathcal{Q} )^2=0$, and $c_2(\mathcal{Q} )(c_1(\mathcal{E}) +  \varepsilon  A)^{n-2}=0$.
By the same argument as above, we have $\codim (\Supp(\mathcal{Q} ^{**}/\mathcal{Q} )) \ge 3$.
The induction hypothesis shows that $\mathcal{Q} ^{**}$ is also a nef vector bundle
with $c_2(\mathcal{Q} ^{**})=0$.

The sheaf $\mathcal{E}$ is locally free by Lemma \ref{Technical_Lemma}.
Furthermore, we have the following exact sequence: 
$$
0 \rightarrow  \mathcal{E}_1 \rightarrow  \mathcal{E} \rightarrow \mathcal{Q} ^{**} \rightarrow 0.
$$
Therefore, the vector bundle $\mathcal{E}$ is nef since both $\mathcal{E}_1$ and $\mathcal{Q} $ are nef.
Since both $c_1(\mathcal{E}_1)$ and $c_1(\mathcal{Q} ^{**} ) $ are numerically proportional to $c_1(\mathcal{E})$, we have
$$
c_2(\mathcal{E}) = c_2(\mathcal{E}_1) + c_1(\mathcal{E}_1)c_1(\mathcal{Q} ^{**} ) + c_2(\mathcal{Q} _{2}^{**} ) =0.
$$

\smallskip

\textit{Case of $\nu(c_1(\mathcal{E}))= 0$.}
The sheaf $\mathcal{E} = \mathcal{E}^{**}$ is a numerically flat vector bundle by Theorem \ref{Ou_chapter6},
which finishes the proof.
\end{proof}

\begin{rem}
\label{rem-filtration-subbundle}
By the proof of Theorem \ref{generically_nef_reflexive} (2), if $\mathcal{E}$ is reflexive, $c_1(\mathcal{E})$ is nef, if $\mathcal{E}$ is generically nef, and if $c_2(\mathcal{E})A^{n-2}  = 0$, then there exists a filtration of subbundles: 
$$
0=\colon E_0 \subset E_1 \subset \cdots \subset E_l:=\mathcal{E}
$$
such that $E_i / E_{i-1}$ is projectively numerically flat and $c_1 (E_i /E_{i-1})$ is numerically proportional to $c_1(\mathcal{E})$.
In particular, each quotient $E_{i}/E_{i-1}$ is nef, and $E$ is an extension of projectively numerically flat bundles.
\end{rem}

Theorem \ref{generically_nef_reflexive} (2) leads to the following corollary.

\begin{cor}
\label{cotangent_nef_proj}
Let $X$ be a smooth projective variety, and let $A$ be an ample line bundle on $X$.
\begin{enumerate}
\item[$(1)$] If $K_X$ is nef and $c_2(X) A^{n-2} =0$, then $\Omega_{X}$ is nef and $c_2(X)=0$ in $H^{2,2}(X,\R)$.
\item[$(2)$] If $-K_X$ is nef and $c_2(X) A^{n-2} =0$, then $T_{X}$ is nef and $c_2(X)=0$ in $H^{2,2}(X,\R)$.
\end{enumerate}
\end{cor}
\begin{proof}
(1) 
Assume that $K_X$ is nef and $c_2(X) A^{n-2} =0$. By \cite{Miyaoka87} (cf.\,\cite[Theorem 0.3]{CP11} or \cite[Theorem 1.3]{CP15}),
the cotangent bundle $\Omega_{X}$ is generically nef. Hence, Conclusion (1) follows from Theorem \ref{generically_nef_reflexive}.

\smallskip 
(2)
Assume that $-K_X$ is nef and $c_2(X) A^{n-2} =0$. 
By \cite[Theorem 1.4]{Ou17}, the tangent bundle $T_X$ is  generically nef. 
Hence, Conclusion (2) follows from Theorem \ref{generically_nef_reflexive}.
\end{proof}

\subsection{Case of compact K\"ahler manifolds}
Theorems \ref{Miyaoka_second_chern} and \ref{Ou_chapter6}
can be generalized to compact K\"ahler manifolds.
See  Section \ref{section:appendix} for some remarks for the Chern  classes of coherent sheaves 
on compact K\"ahler manifolds. 
The proof is essentially the same as in \cite{Cao13};
hence, we put the proof in Appendix \ref{section:appendix}.
By using Theorem \ref{Cao13_revisit} instead of Theorem \ref{Ou_chapter6},
we can repeat the same argument as in Theorem \ref{generically_nef_reflexive} (2)
for compact K\"ahler manifolds,
which proves Theorem \ref{generically_nef_reflexive} (1).
Hence, we omit the proof.

\begin{thm} $($cf.\,\cite[Proposition 4.4 and Lemma 4.5]{Cao13}$)$
\label{Cao13_revisit}
Let $(X, \omega)$ be a compact K\"ahler manifold of dimension $n \ge 2$, and let $\mathcal{E}$ be a torsion-free   coherent sheaf.
Assume that $c_1(\mathcal{E})$ is nef and $\mathcal{E}$ is $\{ \eta  \}^{n-1}$-generically nef for any K\"ahler form $\eta$.
Then, $c_2(\mathcal{E})  (c_1(\mathcal{E}) +  \varepsilon  \{\omega\})^{n-2} \ge0$ holds for any $0 <  \varepsilon  \ll 1.$

Moreover, if $\mathcal{E}$ is reflexive and there exists $ \varepsilon _{0}$ depending on $(X, \omega)$ and $\mathcal{E}$
such that
if $c_2(\mathcal{E}) (c_1(\mathcal{E}) +  \varepsilon  \{\omega\})^{n-2} =0$ for some $ \varepsilon _{0} >  \varepsilon >0$,
then the following holds: 

\begin{itemize}
\item[(1)] In the case of $\nu(c_1(\mathcal{E}))\ge 2$,
there exists a line bundle $\mathcal{E}_1 \subset \mathcal{E}$ such that $\mathcal{E}/\mathcal{E}_{1}$ is torsion-free, $c_1(\mathcal{E}_1)=c_1(\mathcal{E})$, $c_1(\mathcal{E}/\mathcal{E}_1)=0$, and $c_2(\mathcal{E}/\mathcal{E}_1) (c_1(\mathcal{E}) +  \varepsilon  \{\omega\})^{n-2}=0$.

\item[(2)] In the case of $\nu(c_1(\mathcal{E}))= 1$,
there exists a filtration of torsion-free coherent sheaves: 
\begin{equation*}
\label{eq-filtraion-HN}
0 = \mathcal{E}_0 \subset \mathcal{E}_1 \subset \cdots \subset \mathcal{E}_l =\mathcal{E}
\end{equation*}
which satisfies the following conditions.
\begin{itemize}
\renewcommand{\theenumi}{\alph{enumi}}
\item[$\bullet$] This filtration is the $(c_1(\mathcal{E}) +  \varepsilon  \{\omega\})^{n-1}$-Harder-Narasimhan filtration.
\item[$\bullet$] $c_1(\mathcal{E}_k /\mathcal{E}_{k-1})$ is nef and numerically proportional to $c_1(\mathcal{E})$ for any $k = 1, \dots, l$.
\item[$\bullet$] $c_2(\mathcal{E}_k /\mathcal{E}_{k-1}) (c_1(\mathcal{E}) +  \varepsilon  \{\omega\})^{n-2}=0$ for any $k = 1, \dots, l$.
\end{itemize}

\item[(3)] In the case of $\nu(c_1(\mathcal{E}))= 0$,
the sheaf $\mathcal{E}$ 
is a numerically flat vector bundle.
\end{itemize}
\end{thm}
In particular, by Remark \ref{rem-filtration-subbundle}, under the assumptions of Theorem \ref{Cao13_revisit},
if $\mathcal{E}$ is reflexive and
$$
c_2(\mathcal{E}) (c_1(\mathcal{E}) + \varepsilon {\omega})^{n-2}=0
$$
for some $0<\varepsilon<\varepsilon_{0}$, then there exists a filtration by subbundles
$$
0=:E_0 \subset E_1 \subset \cdots \subset E_l:=\mathcal{E}
$$
such that each quotient $E_i / E_{i-1}$ is projectively numerically flat and
$c_1(E_i /E_{i-1})$ is numerically proportional to $c_1(\mathcal{E})$.

\begin{cor}\label{cotangent_nef_kahler}
Let $X$ be a compact K\"ahler manifold and let $\omega$ be a K\"ahler form. We take $ \varepsilon _0 >0$ as in Theorem \ref{Cao13_revisit}.
\begin{enumerate}
\item[$(1)$] If $K_X$ is nef and $c_2(X)  (c_1(K_X) +  \varepsilon  \{\omega\})^{n-2} =0$ for some $ \varepsilon _{0} >  \varepsilon >0$, then $\Omega_{X}^{1}$ is nef and $c_2(X)=0$ in $H^{2,2}(X,\R)$.
\item[$(2)$] If $-K_X$ is nef and $ c_2(X) (c_1(-K_X) +  \varepsilon  \{\omega\})^{n-2} =0$ for some $ \varepsilon _{0} >  \varepsilon >0$, then $T_{X}$ is nef and $c_2(X)=0$ in $H^{2,2}(X,\R)$.
\end{enumerate}
\end{cor}
\begin{proof}
(1) 
By \cite[Theorem 1.2]{Cao13} or \cite[Theorem 1.4]{Enoki93},
the cotangent bundle $\Omega_{X}$ is $\{ \eta \}^{n-1}$-generically nef
for any K\"ahler form $\eta$. Hence, Conclusion (1)
follows from Theorem \ref{generically_nef_reflexive}.

\smallskip 
(2) 
By \cite[Theorem 1.2]{Cao13}, the tangent bundle $T_{X}$ is $\{ \eta \}^{n-1}$-generically nef
for any K\"ahler form $\eta$.
Hence, Conclusion (2) follows from Theorem \ref{generically_nef_reflexive}.
\end{proof}

We prove Corollary \ref{Cao_Ou_structure_thm} by using the following proposition: 

\begin{prop}
\label{Fano_c2}
If $X$ is a Fano manifold of dimension $n \geq 2$, then $c_2(X) \neq 0$ in $H^{2,2}(X,\R)$.
\end{prop}
\begin{proof}
Assume that $c_2(X) = 0$ in $H^{2,2}(X,\R)$.
Then, we have $c_2(X) c_1(-K_X)^{n-2}=0$ and $\nu(c_1(-K_X))=n\ge2$.
By the proof of Theorem \ref{generically_nef_reflexive} (2),
there exists a quotient bundle $T_{X} \to E_2 \to 0$ of rank $\dim X -1$
such that $E_2$ is a numerically flat vector bundle.
The vector bundle $E_{2}$ is trivial since $X$ is simply connected.
This implies that  $$h^{0}(X, \Omega_{X}) \geq \dim X - 1>0,$$
which contradicts the rational connectedness of  $X$.
\end{proof}

\begin{proof}[Proof of Corollary \ref{Cao_Ou_structure_thm}]

Since $T_X$ is nef by Corollary \ref{cotangent_nef_kahler},
there exists a finite \'etale morphism $\pi\colon X' \rightarrow X$ and a smooth morphism
$\alpha \colon X' \rightarrow Y'$ such that $Y'$ is a torus and any fiber $F$ of $f$ is Fano
by \cite[Main Theorem]{DPS94}.
By considering the restriction of the standard exact sequence of cotangent bundles to $F$, 
we obtain $c_2(F)=0$ by $c_2(X')=0$. 
Since the fiber $F$ is Fano,
we conclude that $\dim F \le 1$ by Proposition \ref{Fano_c2}.
If $\dim F=0$, then $X' \cong Y'$ is a complex torus.
If $\dim F=1$, then $F \cong \mathbb{P}^1$.
Then $\alpha \colon X' \rightarrow Y$ is a $\mathbb{P}^1$-fiber bundle.
\end{proof}

If $\Omega_{X}$ is nef and $c_{2}(X)=0$,
then $K_{X}=\det \Omega_{X}$ is semi-ample by Theorem \ref{thm-main2}.
However, the following example shows that
this is not true when replacing $\Omega_{X}$ with $T_{X}$.

\begin{ex}
Let $C$ be an elliptic curve, and let $D$ be a divisor on $C$ such that $\deg D=0$ and $D$ is not a torsion element in $\mathrm{Cl}(X)$.
Set $$X:=\mathbb P( \mathcal O_C \oplus \mathcal O_C(-D)).$$
We denote by $\phi\colon X \to C$ the natural projection.
Then, the anticanonical bundle $-K_X = -K_{X/C}$ is nef but not semi-ample
(for example, see \cite[Example 6.2]{EIM21}).
From $c_1(\phi^{*}K_C)=0$, we obtain $c_2(X) = c_1(K_{X/C})c_1(\phi^{*}K_C) =0$.
\end{ex}

\section{Appendix} \label{section:appendix}
\subsection{Proof of Theorem \ref{Cao13_revisit}}

In this subsection, we prove Theorem \ref{Cao13_revisit}.
The proof is essentially the same as in \cite{Cao13},
but we provide the details for the reader's convenience.

We first give some remarks for the Chern classes of coherent sheaves on compact K\"ahler manifolds.
By \cite[Theorem 1 and Corollary 1]{Gri10}, for a coherent sheaf $\mathcal{F}$ on a compact K\"ahler  manifold $X$, we can define the Chern class $c_i(\mathcal{F}) \in H^{2i}(X,\Q)$, which is compatible with the
rational topological Chern class.
Thus, as in \cite[Definition 7.1]{CHP16}, we can define the intersection number $c_2(\mathcal{F})\{ \omega\}^{n-2}$ for a K\"ahler form $\omega$ on $X$.
Further we can use this intersection number in the same way as in the case of projective manifolds 
thanks to results proved in \cite[Chapter 7]{CHP16}.

We now check an elementary result used in \cite[Lemma 6.7]{Ou17} and \cite{Cao13}.
\begin{lem}
\label{linear_algebra}
Let $m$ be a positive integer.
We define the quadratic form by
$$
q(x,y):= x_0 y_0 - (x_1y_1 + \cdots + x_m y_m) 
$$
for $x = (x_0, x_1, \ldots, x_m), y = (y_0, y_1, \ldots, y_m)\in \R^{m+1}$.
Then, we have: 
\begin{enumerate}
\item[$(1)$] If $q(x,x)\ge0$, then $q(x,x)q(y,y) \le q(x,y)^2$.
\item[$(2)$] If $q(x,x)>0$ and $q(x,y)=q(y,y)=0$, then $y=0$.
\item[$(3)$] If $q(x,x)=q(x,y)=q(y,y)=0$, then $x, y$ are linearly dependent.
\item[$(4)$] If $x \neq 0$ and $q(x,y)=q(x,x)=0$, then $q(y,y) \le 0$.
\end{enumerate}
\end{lem}
\begin{proof}
Set $F(t) := q(x+ty, x+ty) = q(x,x)+2tq(x,y) + t^2q(y,y)$.

(1). We may assume that $q(x,x) >0$ and $q(y,y) >0$. From $y_0 \neq0$, we have $F(0) >0$ and $F(-\frac{x_0}{y_0}) \le 0$. Thus, $q(x,y)^2 \ge q(x,x)q(y,y)$ since $F(t)=0$ has a real solution.

(2). If $y_0 \neq 0$, then $0\ge F(-\frac{x_0}{y_0}) =F(0)>0$, which is impossible. Thus, $y_0 =0$. From $q(y,y)=0$, we have $y=0$.

(3). We may assume that $y \neq 0$. From $F(-\frac{x_0}{y_0}) =0$, we have $x - \frac{x_0}{y_0}y=0$.

(4). The proof
follows from $q(y,y)=q(y-\frac{y_0}{x_0}x, y-\frac{y_0}{x_0}x) \le 0$.
\end{proof}

Let $X$ be a compact K\"ahler manifold of dimension $n \ge 2$, and let $\omega$ be a K\"ahler form.
We define $$q(\eta_1, \eta_2 ) = \{\eta_1\} \{\eta_2\} \{\omega\}^{n-2}$$
for $\eta_1, \eta_2 \in H^{1,1}(X, \R)$.
Then, the signature of $q$ is $(1,h^{1,1}(X, \R)-1 )$ (for example, see \cite{Voison}).
Hence, we can apply Lemma \ref{linear_algebra} to $q(\eta_1, \eta_2 )$.

\begin{proof}[{Proof of Theorem \ref{Cao13_revisit}.}]
The conclusion is obvious in the case of $\rk \mathcal{E} =1$.
Hence, we may assume that $\rk \mathcal{E} \ge 2$.

Set $\nu:=\nu(c_1(\mathcal{E}))$ and $\alpha_{ \varepsilon }:=c_1(\mathcal{E}) +  \varepsilon  \{\omega\}$ for any $  \varepsilon  >0$.
By \cite[Proposition 2.3]{Cao13}, if $  \varepsilon  >0$ is small enough,
the $\alpha_{ \varepsilon }^{n-1}$-Harder Narasimhan filtration
$$
0 =\colon \mathcal{E}_0 \subset \mathcal{E}_1 \subset \cdots \subset \mathcal{E}_l:=\mathcal{E}
$$
is independent of $  \varepsilon  $.
Set $\mathcal{G}_i := \mathcal{E}_i / \mathcal{E}_{i-1}$ and $r_i := \rk(\mathcal{G}_i)$.
Since $\mathcal{E}$ is $\{ \eta \}^{n-1} $-generically nef for any K\"ahler form $\eta$, we have $\mu_{\alpha_{ \varepsilon }^{n-1}}(\mathcal{G}_i) \ge 0$.
The sheaf $\mathcal{G}_i$ is an $\alpha_{ \varepsilon }^{n-1}$-semistable sheaf,
and thus the Bogomolov-Gieseker inequality yields
\begin{align}
\begin{split}
\label{first_inequality}
2c_2(\mathcal{E}) \alpha_{ \varepsilon }^{n-2}
&= \left(c_1(\mathcal{E})^2 + \sum_{1\le i\le l}(2c_2(\mathcal{G}_i) -c_1(\mathcal{G}_i)^2)\right)\alpha_{ \varepsilon } ^{n-2} \\
&\ge  \left(c_1(\mathcal{E})^2 - \sum_{1\le i\le l}\frac{c_1(\mathcal{G}_i)^2}{r_i}\right)\alpha_{ \varepsilon } ^{n-2}. \\
\end{split}
\end{align}

\begin{claim}\label{label_c}
$c_1(\mathcal{G}_i)c_1(\mathcal{E})^{\nu}\{\omega\}^{n-1-\nu}  =0$ holds if $\nu\le n-1$.
\end{claim}
\begin{proof}[Proof of Claim \ref{label_c}]
We have $
c_1(\mathcal{G}_i)c_1(\mathcal{E})^{\nu}\{ \omega\}^{n-1-\nu}  \ge0$ from
$$
0 \leq \mu_{\alpha_{ \varepsilon }^{n-1}}(\mathcal{G}_i) = \binom{n-1}{\nu}
\frac{\left( c_1(\mathcal{G}_i)c_1(\mathcal{E})^{\nu}\{\omega\}^{n-1-\nu}\right)}{r_i}  \varepsilon ^{n-1-\nu} + O( \varepsilon ^{n-\nu}). $$
The sum of them can be computed as follows: 
\begin{align*}
\sum_{1\le i \le l} c_1(\mathcal{G}_i)c_1(\mathcal{E})^{\nu}\{\omega\}^{n-1-\nu}
=&\sum_{1\le i \le l} \bigl(c_1(\mathcal{E}_i)-c_1(\mathcal{E}_{i-1})\bigr)c_1(\mathcal{E})^{\nu}\{\omega\}^{n-1-\nu} \\ 
=&c_1(\mathcal{E})^{\nu+1}\{\omega\}^{n-1-\nu}.
\end{align*}
The right-hand side is zero by the definition of the numerical dimension.
This completes the proof.
\end{proof}

The proof is divided into three cases: $\nu\ge 2$, $\nu =1$, and $\nu = 0$.
\smallskip
\\
\textit{Case of $\nu \ge 2$.}
(\ref{first_inequality}) shows that
$c_2(\mathcal{E})\alpha_{ \varepsilon }^{n-2}>0$ holds if $l=1$.
Hence, we may assume that $l\ge2$.
Set $a_i := c_1(\mathcal{G}_i)c_1(\mathcal{E})^{\nu -1}\{\omega\}^{n-\nu}$.
Then, we have
$$
\mu_{\alpha_{ \varepsilon }^{n-1}}(\mathcal{G}_i)=\binom{n-1}{\nu-1} \frac{a_i}{r_i} \varepsilon ^{n-\nu}+O(\varepsilon ^{n-\nu+1}).
$$
From $\mu_{\alpha_{ \varepsilon }^{n-1}}(\mathcal{G}_1) > \cdots > \mu_{\alpha_{ \varepsilon }^{n-1}}(\mathcal{G}_l)\ge0$,
we obtain 
\begin{equation}
\label{eq-HNfil-decreasing}
a_1/r_1 \ge a_2/r_2 \ge \cdots \ge a_l/r_l\ge0.
\end{equation}

\begin{claim}
\label{nu2_inequality}The following estimate holds: 
\begin{align*}
&2c_2(\mathcal{E}) \alpha_{ \varepsilon }^{n-2} \\
\ge &
\binom{n-2}{\nu-2} \left(\sum_{1\le k\le l} a_k\right)^{-1}
\left(\sum_{1\le i \le l} \frac{(r_i - 1)a_{i}^{2}}{r_i} + \sum_{1 \le i<j\le l}2a_i a_j \right) \varepsilon ^{n-\nu} + O( \varepsilon ^{n-\nu+1}).
\end{align*}
In particular, if there exists $a>0$ such that
$$\sum_{1\le i \le l}\frac{(r_i - 1)a_{i}^{2}}{r_i}+ \sum_{1 \le i<j\le l}2a_i a_j \ge a,
$$
then $c_2(\mathcal{E}) \alpha_{ \varepsilon }^{n-2}>0$ for sufficiently small $ \varepsilon >0$ 
by $\sum_{1\le k\le l} a_k>0$.
\end{claim}

\begin{proof}[Proof of Claim \ref{nu2_inequality}]
By Lemma \ref{linear_algebra}, we have
\begin{equation}
\label{Hodge_index}
c_1(\mathcal{G}_i)^2\alpha_{ \varepsilon } ^{n-2} 
\le \frac{(c_1(\mathcal{G}_i) c_1(\mathcal{E})\alpha_{ \varepsilon }^{n-2})^2}{c_1(\mathcal{E})^2\alpha_{ \varepsilon }^{n-2}}.
\end{equation}
Hence, we obtain the following equalities: 
\begin{align*}
\begin{split}
c_1(\mathcal{G}_i)c_1(\mathcal{E}) \alpha_{ \varepsilon }^{n-2} 
&=\binom{n-2}{\nu-2} a_i \varepsilon ^{n-\nu}+ O( \varepsilon ^{n-\nu +1}),  \\
c_1(\mathcal{E})^2\alpha_{ \varepsilon }^{n-2} 
&= \sum_{1\le k\le l}c_1(\mathcal{G}_k)c_1(\mathcal{E})\alpha_{ \varepsilon }^{n-2} \\
&=\binom{n-2}{\nu-2} \left(\sum_{1\le k\le l} a_k \right) \varepsilon ^{n-\nu}+ O( \varepsilon ^{n-\nu +1}),  \\
\left(c_1(\mathcal{E})^2 \alpha_{ \varepsilon }^{n-2}\right)^{-1} 
&=
\left(\binom{n-2}{\nu-2}\left( \sum_{1\le k\le l} a_k \right)\varepsilon ^{n-\nu}+ O( \varepsilon ^{n-\nu +1}) \right)^{-1}\\
&=
\frac{ \varepsilon ^{-n+\nu}}{\binom{n-2}{\nu-2}\left(\sum_{1\le k\le l} a_k \right)}+ O( \varepsilon ^{-n+\nu+1}). \\
\end{split}
\end{align*}
From (\ref{first_inequality}) and (\ref{Hodge_index}), by noting that $\sum_{1\le k\le l} a_k>0$, 
we obtain

\begin{align*}
\begin{split}
& 2c_2(\mathcal{E}) \alpha_{ \varepsilon } ^{n-2} \\
\ge& c_1(\mathcal{E})^2\alpha_{ \varepsilon }^{n-2} - \sum_{1\le i\le l}\frac{(c_1(\mathcal{G}_i)c_1(\mathcal{E}) \alpha_{ \varepsilon }^{n-2})^2}{r_ic_1(\mathcal{E})^2 \alpha_{ \varepsilon }^{n-2}} \\
=& \binom{n-2}{\nu-2} \left(\sum_{1\le k\le l} a_k \right)^{-1}
\left(\sum_{1\le i\le l} a_i \sum_{1\le j\le l} a_j -\sum_{1\le i\le l} \frac{a_{i}^{2}}{r_i}\right)\varepsilon ^{n-\nu}
+ O( \varepsilon ^{n-\nu+1}) \\
=& 
\binom{n-2}{\nu-2} \left(\sum_{1\le k\le l} a_k \right)^{-1}
\left(\sum_{1\le i \le l} \left(1 - \frac{1}{r_i}\right)a_{i}^{2} + \sum_{1 \le i<j\le l}2a_i a_j\right)\varepsilon ^{n-\nu} + O( \varepsilon ^{n-\nu+1}). \\
\end{split}
\end{align*}
\end{proof}

If $a_2>0$ or $r_1 \neq 1$, we can find that there exists $a>0$ such that
$$
\sum_{1\le i \le l} \frac{(r_i - 1)a_{i}^{2}}{r_i} + \sum_{1 \le i<j\le l}2a_i a_j \ge a
$$
from $a_1/r_1 \ge a_2/r_2 \ge  \cdots \ge a_l/r_l\ge0$.
Hence, by Claim \ref{nu2_inequality}, we can conclude that
$c_2(\mathcal{E})\alpha_{ \varepsilon }^{n-2}>0$ for sufficiently small $ \varepsilon >0$.

From now on, we may assume that $a_2=0$ and $r_1= 1$.
Then, from (\ref{first_inequality}) and $r_1=1$, by using $c_1(\mathcal{G}_1)= c_1(\mathcal{E}) - \sum_{2 \le i \le l} c_1(\mathcal{G}_i)$, we obtain
\begin{align}
\label{inequality_onezero}
2c_2(\mathcal{E})\alpha_{ \varepsilon }^{n-2}   
&\ge
2c_1(\mathcal{G}_1)\left ( \sum_{2 \le i \le l} c_1(\mathcal{G}_i) \right) \alpha_{ \varepsilon }^{n-2}
+
\left ( \sum_{2 \le i \le l} c_1(\mathcal{G}_i) \right)^{2} \alpha_{ \varepsilon }^{n-2}
-   \sum_{2 \le i \le l} \frac{c_1(\mathcal{G}_i)^2 \alpha_{ \varepsilon }^{n-2}}{r_i}.
\\ 
&= 2 \sum_{2 \le i \le l} c_1(\mathcal{G}_i) c_1(\mathcal{E})\alpha_{ \varepsilon }^{n-2}
-  \left ( \sum_{2 \le i \le l} c_1(\mathcal{G}_i) \right)^{2} \alpha_{ \varepsilon }^{n-2}
-   \sum_{2 \le i \le l} \frac{c_1(\mathcal{G}_i)^2 \alpha_{ \varepsilon }^{n-2}}{r_i}.\notag
\end{align}

Let us consider the case where
$c_1(\mathcal{G}_2)c_1(\mathcal{E})^{\nu -t}\{\omega\}^{n-1-\nu+t}\neq 0 $ holds
for some $t \in \{ 2, \ldots, \nu-1\}$.
Then, we take the minimal number $s \in \{ 2, \ldots, \nu-1\}$ such that 
$$c_1(\mathcal{G}_2)c_1(\mathcal{E})^{\nu -s}\{\omega\}^{n-1-\nu+s}\neq 0, $$ and
set $$b_i := c_1(\mathcal{G}_i)c_1(\mathcal{E})^{\nu -s}\{\omega\}^{n-1-\nu+s}$$
for any $i = 2, \ldots, l$.
For \(i>2\), the Harder--Narasimhan filtration gives \(\mu_{\alpha_{\varepsilon}^{n-1}}(\mathcal{G}_2)>\mu_{\alpha_{\varepsilon}^{n-1}}(\mathcal{G}_i)\), and the generic nefness gives \(\mu_{\alpha_{\varepsilon}^{n-1}}(\mathcal{G}_i)\ge 0\). 
Hence, by the same argument used to derive \eqref{eq-HNfil-decreasing}, we have \(c_1(\mathcal{G}_i)c_1(\mathcal{E})^{\nu-t}\omega^{n-1-\nu+t}=0\) for every \(t=1,\ldots,s-1\). This yields 
\begin{equation}
\label{eq-non-negative-bi}
0 \le  c_1(\mathcal{G}_i) \alpha_{ \varepsilon }^{n-1} = \binom{n-1}{\nu-s} b_i \varepsilon ^{n-\nu+s-1} + O(\varepsilon ^{n-\nu+s})
\end{equation}
for any \(i \ge 2\). If \(b_2<0\), this inequality \eqref{eq-non-negative-bi} is violated for sufficiently small \(\varepsilon\), thus we obtain \(b_2>0\) since \(b_2 \neq 0\) by the choice of \(s\).
Using the above expansions and Lemma \ref{linear_algebra}, we obtain the following equalities: 
\begin{align*}
\begin{split}
c_1(\mathcal{G}_i) c_1(\mathcal{E}) \alpha_{ \varepsilon }^{n-2}
&=
\binom{n-2}{\nu-s-1} b_i \varepsilon ^{n-\nu+s-1}  + O( \varepsilon ^{n-\nu+s}),  \\
\left ( \sum_{2 \le i \le l} c_1(\mathcal{G}_i) \right)^{2} \alpha_{ \varepsilon }^{n-2}
& \le \frac{ \left(\sum_{2 \le i \le l} c_1(\mathcal{G}_i)c_1(\mathcal{E})\alpha_{ \varepsilon }^{n-2} \right)^2}{c_1(\mathcal{E})^2\alpha_{ \varepsilon }^{n-2}}\\
&= \frac{\binom{n-2}{\nu-s-1}^2 (\sum_{2 \le i\le l}b_{i})^{2} }{ \binom{n-2}{\nu-2}a_1}  \varepsilon ^{n-\nu + 2s -2} + O( \varepsilon ^{n-\nu + 2s -1}), \\
c_1(\mathcal{G}_i)^2 \alpha_{ \varepsilon } ^{n-2}
&\le \frac{\left(c_1(\mathcal{G}_i)c_1(\mathcal{E}) \alpha_{ \varepsilon }^{n-2}\right)^2}{c_1(\mathcal{E})^2 \alpha_{ \varepsilon }^{n-2}}
=
\frac{\binom{n-2}{\nu-s-1}^2 b_{i}^{2} }{ \binom{n-2}{\nu-2}a_1}  \varepsilon ^{n-\nu + 2s -2} + O( \varepsilon ^{n-\nu + 2s -1}).
\end{split}
\end{align*}
From $s-1>0$ and (\ref{inequality_onezero}),
$$2c_2(\mathcal{E})\alpha_{ \varepsilon }^{n-2} \ge 2\binom{n-2}{\nu-s-1} \left(\sum_{2 \le i\le l}b_{i}\right)  \varepsilon ^{n-\nu+s-1}+ O( \varepsilon ^{n-\nu+s}).$$
By \eqref{eq-non-negative-bi}, we have \(b_i\geq 0\) for every \(i>2\).
Since \(b_2>0\), we have \(\sum_{2\leq i\leq l} b_i>0\). Hence the leading
term on the right-hand side is positive, and so
\(c_2(\mathcal{E})\alpha_{\varepsilon}^{n-2}>0\) for sufficiently small
\(\varepsilon>0\).

From now on, we may assume that $c_1(\mathcal{G}_2)c_1(\mathcal{E})^{\nu -t}\{ \omega\}^{n-1-\nu+t} = 0 $ for any $t= 2, \ldots, \nu-1$.
Since we now assume \(a_2=0\), it follows from \eqref{eq-HNfil-decreasing} that 
$$
a_i=c_1(\mathcal{G}_i)c_1(\mathcal{E})^{\nu-1}\{\omega\}^{n-\nu}=0
$$ 
for any \(i\geq 2\).
Thus, by the same method used to derive \eqref{eq-HNfil-decreasing}, we obtain 
$$c_1(\mathcal{G}_i)c_1(\mathcal{E})^{t}\{\omega\}^{n-1-t} = 0 $$
for any $t = 1, \dots,  n-1$ and $i = 2, \dots, l$, and finally
\begin{equation}
\label{inner_product_zero}
c_1(\mathcal{G}_i)c_1(\mathcal{E})\alpha_{ \varepsilon } ^{n-2}=c_1(\mathcal{G}_i)c_1(\mathcal{E})(c_1(\mathcal{E}) +  \varepsilon  \{\omega\})^{n-2}=0.
\end{equation}
From $c_1(\mathcal{E})^2\alpha_{ \varepsilon } ^{n-2} >0$,
we obtain $$c_1(\mathcal{G}_i)^2 \alpha_{ \varepsilon } ^{n-2} \le 0$$ by Lemma \ref{linear_algebra}.
By the same argument as above,
we obtain $\left(\sum_{2 \le i \le l}c_1(\mathcal{G}_i)\right)^2 \alpha_{ \varepsilon } ^{n-2} \le 0$, and finally
$c_2(\mathcal{E})\alpha_{ \varepsilon } ^{n-2} \ge0$ for sufficiently small $ \varepsilon >0$ from (\ref{inequality_onezero}).
Hence, we can take $\varepsilon_0 >0$ such that $c_2(\mathcal{E})\alpha_{ \varepsilon } ^{n-2} \ge0$ for any $\varepsilon_{0} > \varepsilon >0$.

If $c_2(\mathcal{E})\alpha_{ \varepsilon } ^{n-2} =0$ for some $\varepsilon_0> \varepsilon >0$, then $c_1(\mathcal{G})^2 \alpha_{ \varepsilon } ^{n-2} = 0$, where $\mathcal{G} := \mathcal{E}/\mathcal{E}_1$.
By Lemma \ref{linear_algebra},
$c_1(\mathcal{G})=0$ and $c_1(\mathcal{E}_1) = c_1(\mathcal{E})$.
Furthermore, $\mathcal{E}_1$ is a nef line bundle by $r_1=1$.
Hence, we obtain that
$$c_2(\mathcal{G})\alpha_{ \varepsilon } ^{n-2}=c_2(\mathcal{E})\alpha_{ \varepsilon } ^{n-2} - c_1(\mathcal{G})c_1(\mathcal{E}_1)\alpha_{ \varepsilon } ^{n-2}=0.$$
\\
\textit{Case of $\nu = 1$.}
By Claim \ref{label_c}, we have
\begin{equation}
\label{vanishing_first}
c_1(\mathcal{G}_i)c_1(\mathcal{E})\alpha_{ \varepsilon }^{n-2}=c_1(\mathcal{E})^2\alpha_{ \varepsilon }^{n-2}=0.
\end{equation}
Furthermore, from $c_1(\mathcal{E})\neq0$ and Lemma \ref{linear_algebra},
we have
$c_1(\mathcal{G}_i)^{2}\alpha_{ \varepsilon }^{n-2} \le 0$.
Together with (\ref{first_inequality}), we obtain
$$
2c_2(\mathcal{E}) \alpha_{ \varepsilon }^{n-2} \ge  - \sum_{1\le i\le l}\frac{c_1(\mathcal{G}_i)^2}{r_i}\alpha_{ \varepsilon } ^{n-2} \ge 0.
$$
If $c_2(\mathcal{E})\alpha_{ \varepsilon }^{n-2}=0$ holds for some $ \varepsilon >0$, then $c_1(\mathcal{G}_i)^{2}\alpha_{ \varepsilon }^{n-2} =0$.
By Lemma \ref{linear_algebra} and (\ref{vanishing_first}),
$c_1(\mathcal{G}_i)$ and $ c_1(\mathcal{E})$ are linearly dependent.
This implies that $c_1(\mathcal{G}_i)c_1(\mathcal{G}_j)\alpha_{ \varepsilon }^{n-2}=0$ for any $1 \le i\le j \le l$.
Since $\mathcal{G}_i$ is $\alpha_{ \varepsilon }^{n-1}$-semistable,
we obtain $c_2(\mathcal{G}_i)\alpha_{ \varepsilon }^{n-2} \ge 0$
from the Bogomolov-Gieseker inequality.
Furthermore, from
$$
c_2(\mathcal{E})\alpha_{ \varepsilon }^{n-2}=\sum_{1 \le i \le l} c_2(\mathcal{G}_i)\alpha_{ \varepsilon }^{n-2} =0,
$$
we obtain $c_2(\mathcal{G}_i)\alpha_{ \varepsilon }^{n-2} = 0$ for any $i = 1, \dots, l$.
\\
\textit{Case of $\nu = 0$.}
From $c_1(\mathcal{E})=0$,
$\mathcal{E}$ is $\alpha_{ \varepsilon }^{n-1}$-semistable.
By the Bogomolov-Gieseker inequality, we obtain $c_2(\mathcal{E})\alpha_{ \varepsilon }^{n-2}\ge 0$.
If $c_2(\mathcal{E})\alpha_{ \varepsilon }^{n-2}=0$,
then $c_2(\mathcal{E}^{**})\alpha_{ \varepsilon }^{n-2}=0$.
By Theorem-Definition \ref{numerically_projectively_flat},
the reflexive hull $\mathcal{E}^{**}$ is a numerically projectively flat vector bundle;
thus, $\mathcal{E}^{**}$ is numerically flat by $c_1(\mathcal{E}^{**})=0$.
\end{proof}

\subsection{Related open problems}
\label{subsec-rela}

In this subsection, we collect some problems related to generically nef vector bundles.
\begin{conj}
\label{Miyaoka_thm}
Let $X$ be a compact K\"ahler manifold, let $\mathcal{E}$ be a torsion-free coherent sheaf, and let $\alpha_{1}, \ldots, \alpha_{n-2} $ be K\"ahler classes on $X$.
\\
$(1)$ If $c_1(\mathcal{E})$ is nef and $\mathcal{E}$ is generically nef, then the inequality $c_2(\mathcal{E}) \alpha_{1}  \cdots \alpha_{n-2} \ge 0$ holds.
\smallskip
\\
$(2)$ Let $\beta$ be a nef class. We further assume that
$E$ is $\alpha_{1}  \cdots \alpha_{n-2}\beta$-semistable.
Then, the following inequality holds: 
$$
\left(c_2(\mathcal{E}) - \frac{r-1}{2r}c_1(\mathcal{E})^2 \right)\alpha_{1}  \cdots \alpha_{n-2} \ge 0.
$$
\end{conj}

Note that Conjecture \ref{Miyaoka_thm} (2) leads to Conjecture \ref{Miyaoka_thm} (1)
by the argument of \cite[Theorem 6.1]{Miyaoka87}.
Conjecture \ref{Miyaoka_thm} has been affirmatively solved
when $X$ is projective and both $\alpha_{i}$ and $\beta$ are in the N\'eron-Severi group of $X$.

The following conjecture has also been affirmatively solved in the projective case
but is still open in the compact K\"ahler case.
This conjecture is more difficult than we expected
due to the lack of the foliation theory developed in \cite{CP15} on compact K\"ahler manifolds.

\begin{conj}[{cf.\,\cite[Conjecture 4.3]{Cao13}}]
\label{mixed_kahler}
Let $X$ be a compact K\"ahler manifold, and let $\alpha_{1}, \ldots, \alpha_{n-1} $ be K\"ahler classes.
\begin{enumerate}
\item[$(1)$]
If $K_X$ is nef, then $\Omega_{X}$ is $\alpha_{1}  \cdots \alpha_{n-1}$-generically nef.
\item[$(2)$] If $-K_X$ is nef, then $T_{X}$ is $\alpha_{1}  \cdots \alpha_{n-1}$-generically nef.
\end{enumerate}
\end{conj}





\end{document}